\documentclass[12pt, twoside]{article}
\usepackage{amsmath,amsthm,amssymb}
\usepackage{times}
\usepackage{enumerate}

\theoremstyle{definition}
\newtheorem{thm}{Theorem}[section]
\newtheorem{lem}[thm]{Lemma}
\newtheorem{defin}[thm]{Definition}
\newtheorem{rem}[thm]{Remark}

\numberwithin{equation}{section}

\newcommand{\subjclass}[1]{\bigskip\noindent\emph{2010 Mathematics Subject Classification:}\enspace#1}
\newcommand{\keywords}[1]{\noindent\emph{Keywords:}\enspace#1}

\newtheorem{assumption}[thm]{Assumption}

\providecommand{\charf}[1]{\chi_{#1}}
\providecommand{\Naturals}{\mathbb{N}}
\providecommand{\Reals}{\mathbb{R}}
\providecommand{\set}[1]{\{#1\}}
\providecommand{\Set}[1]{\left\{#1\right\}}

\providecommand*{\diff}[2][]{\frac{\ud{#1}}{\ud{#2}}}
\DeclareMathOperator{\Div}{\mathrm{div}}
\providecommand{\Laplace}{\Delta}
\providecommand*{\ndiff}[1][]{\pdiff[#1]{\normal}}
\providecommand*{\pdiff}[2][]{\frac{\partial #1}{\partial #2}}
\providecommand*{\ud}[1]{\mathop{\mathrm{d}#1}}

\providecommand{\bdry}{\partial}
\providecommand{\Hausdorff}{\mathcal{H}}
\providecommand{\Lebesgue}{\mathcal{L}}
\providecommand{\normal}{\mathbf{n}}

\providecommand{\downto}{\downarrow}
\providecommand{\weaklyto}{\rightharpoonup}
\providecommand{\weakstar}{weak*}
\providecommand{\weakstarto}{\overset{\ast}{\rightharpoonup}}

\providecommand{\eps}{\varepsilon}
\providecommand{\BV}{\mathnormal{BV}}
\providecommand{\loc}{\mathnormal{loc}}

\providecommand{\id}{\mathrm{id}}
\DeclareMathOperator{\tr}{\mathrm{tr}}

\providecommand{\diffusion}{\kappa}
\providecommand{\osmosis}{\beta}
\providecommand{\surfacetension}{\gamma}
\providecommand{\cells}{\mathcal{C}}
\providecommand{\Cells}{\cells_n}
\providecommand{\Energy}{\Phi}
\providecommand{\Entropy}{\mathcal{F}}
\providecommand{\integrand}{f}
\providecommand{\perimeter}{P}
\providecommand{\Perimeter}{\perimeter_n}
\providecommand{\Profiles}{L^1_n}
\providecommand{\setdist}{\mathbf{d}}
\providecommand{\states}{\mathcal{X}}
\providecommand{\States}{\states^n}

\begin{document}

%%%%% To ease editing, add:

%\baselineskip=17pt

%%%%%%%%%%%%%%%%

\title{Cell Swelling by Osmosis: a Variational Approach}

\author{Martijn Zaal\\
VU University Amsterdam\\
m.m.zaal@vu.nl}

\date{}

\maketitle

%%%%%%%%

\begin{abstract}
A very simple model for cell swelling by osmosis is introduced, resulting in a parabolic free boundary problem. In case of radially symmetric initial conditions, it is shown that the model can be viewed as a gradient flow involving entropy, surface area and the Wasserstein metric. This observation is used to construct solutions and explain the presence and nature of osmosis.


\end{abstract}

\section{Introduction}

In this article, a very simple model for osmotic cell swelling is introduced, and studied using a gradient flow approach. Observations of osmotic effects indicate that cell membranes are permeable to water, but impermeable to other molecules. Therefore, it is believed that water transport through cell membranes is facilitated by proteins acting as water channels, generally called `aquaporins' (\cite{Chrispeels_Maurel}). Although there is growing evidence that aquaporins are proteins embedded in the cell membrane (\cite{Verkman}), existence has been shown for only a few types of cells. The exact nature of aquaporins remains uncertain, although aquaporins are believed to be proteins (\cite{Logee_Verkman_Zhang}). 

\subsection{A model for osmotic cell swelling}

Quantitative modeling of water flux through membranes by aquaporins seems to have started with \cite{Logee_Verkman_Zhang}. In this paper, the water flux was assumed to be proportional to the difference in molar free energy inside and outside the cell. However, the model did not predict the eventual slowing of cell swelling. Several other models using this assumption have been considered since, but no model succeeded to provide evidence for either the permeability or aquaporin model when comparing predictions and observations. Furthermore, most models assumed spatially constant concentrations inside and outside the cell.

Recently, Pickard (\cite{Pickard}), made a first attempt to develop a model for cell swelling, based on observations of Xenopus oocytes submitted to hypoosmotic shocks, that is, sudden exposure to low concentrations of solute. The proposed model uses proportionality of free energy differences and water flux as a starting point. In this model, the surface tension of the cell membrane is not taken into account. 

Unfortunately, the justification of dropping surface tension does not seem to be valid if the cell is close to equilibrium. Although it should be noted that in most test tube situations the cell will burst long before it reaches equilibrium, the implications of assuming zero surface tension should be investigated. Furthermore, it might be insightful to compare a mathematical model with aquaporins with a similar model where the cell membrane has a constant permeability. Except for the last section, the constant permeability model will be considered.

\subsection{Mathematical formulation}

The model consists of two main ingredients: a time-varying domain with diffusion of mass inside. 
\begin{equation}\label{eq:diffusion}
	\pdiff[u]{t}=\diffusion\Laplace u.
\end{equation}
The free boundary is non-permeable to the diffusing mass, which is basically a Neumann condition for a free boundary:
\begin{equation}\label{eq:moving_neumann}
	-\diffusion\ndiff[u(t)]=u v.
\end{equation}
Here, $u(x, t)$ is the concentration of solute, and $v$ is the normal velocity of the cell membrane, where the normal $\normal$ is pointing outward. It is comparable to the standard Neumann boundary condition since it requires that there is no flux of solute through the boundary of the free domain. This can be seen by introducing the additional variable $\mathbf{v}(x, t)$, the velocity of the diffusing mass. Since \eqref{eq:diffusion} implies that $\mathbf{v}=-\diffusion\frac{\nabla u}{u}$, \eqref{eq:moving_neumann} is equivalent to
\begin{equation}\label{eq:moving_neumann_system}
	u\mathbf{v}\cdot\normal=uv,
\end{equation}
that is, normal velocity of the boundary and the normal component of the velocity of mass at the boundary should match. Note that this condition \eqref{eq:moving_neumann} does not depend on how the membrane reacts to $u$. 

In the present model, the normal velocity of the membrane is caused by the absorption of an incompressible solvent, usually water. Assuming constant permeability for the cell membrane, the amount of solvent absorbed through a portion of surface is proportional to the difference in the Helmholtz free energy $A$ inside and outside the cell, which is given by
\begin{equation}
	\Delta A=\Delta P+\Delta u,
\end{equation}
where $\Delta P$ and $\Delta u$ denote the pressure difference and concentration difference accross the cell membrane, respectively. The pressure difference is due to the surface tension of the cell membrane, resulting in a term proportional to the mean curvature of the membrane. Since it is assumed that there is no solute outside the cell, the concentration difference is equal to just the value of $u$ at the boundary. Since the solute is assumed to be incompressible, the normal velocity of the boundary should be proportional to the amount of solvent absorbed by the cell. Combined with the above model for the absorption of sovlent, this resulting equation for the normal velocity of the cell membrane is
\begin{equation}\label{eq:boundary_movement}
	v=\surfacetension H+\osmosis u,
\end{equation}
in which $\surfacetension$ and $\osmosis$ are positive parameters. The sign convention for $H$ is according to $\normal$, which means that $H=-\frac{n-1}{r}$ for a sphere of radius $r$ in $\Reals^n$. Note that the first term will tend to make $v$ smaller, whereas the second term tends to make $v$ bigger. This is in accordance with the intuition of the model: the surface tension of the cell membrane will tend to make the cell smaller, whereas the osmotic force tends to make the cell bigger.

A number of parameters from this problem can be eliminated by changing variables. Looking at the units of $u$, $x$ and $t$, it follows that $\diffusion$ and $\osmosis$ have the same dimension. Therefore, at most one of the two can be set equal to $1$ in general by changing units.

Since the total mass is a conserved quantity, it can also be regarded as a parameter, say $\vartheta$. Changing units for $x$ first, the total mass can can be set equal to $1$. Next, changing units for $u$ and $x$ simultaniously, $\osmosis$ can be made equal to $\surfacetension$ without changing the total mass. Finally, scaling $t$, $\surfacetension=\osmosis=1$ can be realized. This results in new, scaled variables
\begin{equation}\begin{aligned}
	\xi &						=\left(\frac{\osmosis}{\vartheta\surfacetension}\right)^\frac{1}{n+1}x,\qquad &
	w(\xi, \tau) &	=\left(\frac{\surfacetension}{\osmosis}\right)^\frac{n}{n+1}\frac{1}{\vartheta^\frac{1}{n+1}}u(x, t), \\
	\tau &					=\left(\frac{\osmosis}{\vartheta}\right)^\frac{2}{n+1}\surfacetension^\frac{n-1}{n+1} t,\qquad &
	\rho(\tau) &		=\left(\frac{\osmosis}{\vartheta\surfacetension}\right)^\frac{1}{n+1}r(t).
\end{aligned}\end{equation}
It is easy to check that $w$ and $\rho$ satisfy \eqref{eq:diffusion}, \eqref{eq:moving_neumann}, \eqref{eq:boundary_movement}, with coefficients $\surfacetension=1$ and $\osmosis=1$, and $\int w\ud{\xi}=1$. Note that the change of coordinates also implies a change in the coefficient $\kappa$.

It follows that one can restrict attention to the problem
\begin{equation}\label{eq:osmosis_scaled}
	\left\{\begin{aligned}
		\pdiff[u]{t} &					=\diffusion\Laplace u, &\qquad&	\text{for $x\in\Omega(t)$, $t>0$}, \\
		-\diffusion\ndiff[u] &	=u v, &&												\text{for $x\in\bdry\Omega(t)$, $t>0$}, \\
		v &											=H+u, && 												\text{for $x\in\bdry\Omega(t)$, $t>0$},
	\end{aligned}\right.
\end{equation}
with the integral of $u$ equal to 1. The remaining parameter $\kappa$, which cannot be scaled out, can now be regarded as a measure of how fast the dynamics is inside the domain, compared to the movement of the boundary.\bigskip

Looking at \eqref{eq:osmosis_scaled}, a few observations can be made. First of all, the diffusion equation can be rewritten as a system of two equations with two unknowns:
\[
	\left\{\begin{aligned}
		\mathbf{v} &=-\diffusion\frac{\nabla u}{u}, \\
		\pdiff[u]{t}&=-\Div(u\mathbf{v}).
	\end{aligned}\right.
\]
Since $\mathbf{v}$ is regarded as the velocity of mass particles at a certain position and time, the first equation is, up to multiplication by $u$, Fick's law of diffusion, which can be regarded as a modelling law describing how mass moves in reaction to the spatial variations of the concentration $u$. In constrast, the second equation is conservation of mass.
The boundary condition, which was given in terms of $\mathbf{v}$ in \eqref{eq:moving_neumann_system}, reduces to the requirement that the normal part of the velocity of the mass and the velocity of the boundary should math whenever $u$ is positive. From a modeling point of view this formulation as a system is more natural: each equation and each unknown has a physical interpretation. In the classic approach, one of course eliminates the velocity and writes the diffusion equation with only $u$ in it. This representation as a system representation can also be found, for instance, in \cite[\S2]{Otto}. Here, the Porous Media Equation is reformulated in the same way, by replacing \emph{only} Fick's law by another law.

A similar observation can be made for the time-varying domain. The splitting between a modeling equation and a continuity equation has already been done. The modeling equation is, in this case, the  expression for the normal velocity $v$. The continuity part is not written explicitly, but in words, it is the remark that $v$ is the normal velocity of $\bdry\Omega(t)$. This can be put into a formula in a number of ways, depending on the smoothness of $\bdry\Omega(t)$.

Summing up, the problem \eqref{eq:osmosis_scaled} can be viewed as a system of \emph{four} relations, each with a physical interpretation:
\begin{itemize}
\item	Fick's law of diffusion, giving an expression for the velocity of particles,
\item	Conservation of mass, relating the change in $u$ to the velocity,
\item	The modelling law $v=H+u$,
\item	$v$ being the normal velocity of $\bdry\Omega$.
\end{itemize}
As explained above, these four relations can be divided in two groups in two different ways: the first two describe the evolution of $u$, whereas the latter describe $\Omega$. On the other hand, one can argue that the first and third contain the actual model that is being studied, whereas the second and fourth are merely equations relating the evolution of $u$ and $\Omega$ to the functions $\mathbf{v}$ and $v$, respectively. These four equations will play a key role in what follows, as will the two subdivisions made here, which are natural from a modeling point of view.\bigskip

Note that the problem is \emph{almost} a superposition of two well-known problems: the diffusion equation and the mean curvature flow. However, there is a nonlinear coupling, which consists of two parts. First of all, $u$ is required to be supported inside $\Omega$, and secondly an osmotic boundary term appears in the evolution of $\Omega$. It will turn out that there is a close relation between these two couplings; in some sense, the latter follows from the former, as will be shown later.

For both the diffusion equation and the mean curvature flow, variational approaches have been developed. A well-known paper about a variational approach to the Fokker-Planck equation, and the diffusion equation in particular, is \cite{Jordan_Kinderlehrer_Otto}. It has turned out that this is a fundamental example of a flow that can be seen as a gradient flow in a metric space, for which a fairly general theory is developed in \cite{Ambrosio_Gigli_Savare}, with gradient flows in the space of probability measures, and diffusion in particular as an example. Beside the diffusion equation, many parabolic equations have been modeled as gradient flows in the space of probability measures with the Wasserstein metric, see for instance \cite{Otto}, which also treats the Porous Medium Equation, or \cite{Agueh}.

A variational approach for the mean curvature flow can be found, for example, in \cite{Almgren_Taylor_Wang} or \cite{Luckhaus_Sturzenhecker}. Although there is, a far as the author is aware, no theorem stating that the mean curvature flow is in fact a gradient flow in a metric space, the similarities between the discretization used in \cite{Almgren_Taylor_Wang} and \cite{Luckhaus_Sturzenhecker} on the one hand, and the discretization used for the diffusion equation in \cite{Jordan_Kinderlehrer_Otto} on the other hand, strongly suggest that the mean curvature flow is a gradient flow. The functional playing the role of the metric in \cite{Almgren_Taylor_Wang} is, however, not a metric. Recalling the first variation of area formula, an appropriate metric would be related to the $L^2$ norm of the normal velocity, as is noted in \cite{Grunewald_Kim}. Unfortunately, this yields a trivial metric in general.

One way to avoid this problem and obtain a gradient flow with respect to an honest metric, is to restrict attention to the case of radial symmetry. Although this essentially reduces the boundary movement to a one-dimensional problem, the resulting problem is still rich enough to see how exactly the osmotic term in the normal velocity $v$ arises from the restriction on the support of $u$.

Another possible approach to the problem is to combine the time discretization techniques in \cite{Luckhaus_Sturzenhecker} and \cite{Jordan_Kinderlehrer_Otto} without making an explicit connection to gradient flows. Obviously, this is technically much more involved. Although some of the results are easily generalized, taking the limit for the discretization parameter to zero is much harder than in either of the two individual cases.

In the following sections, it will be shown that the radially symmetric version of \eqref{eq:osmosis_scaled} is a gradient flow in a metric space, and solutions will be constructed using this observation. More precisely, in Section \ref{s:GMM}, a metric space and functional will be constructed for radially symmetric initial conditions, and fitted into the framework of \cite{Ambrosio_Gigli_Savare}, enabling the use of a theorem that guarantees the convergence of an abstract version of the approach used in \cite{Jordan_Kinderlehrer_Otto}, \cite{Luckhaus_Sturzenhecker} and \cite{Almgren_Taylor_Wang}. In Section \ref{s:convexity}, some finer properties of the space and functional from Section \ref{s:GMM} will be studied, resulting in a theorem concerning, among other things, regularity of the curve constructed in Section \ref{s:GMM}. Convexity of certain functionals will be a special point of interest. Section \ref{s:MaxSlope} is devoted to the study of the differential properties of the functional constructed in Section \ref{s:GMM}, making the conclusions of Theorem \ref{th:convexity_results} more concrete. In Section \ref{s:WeakSol} it is shown, using theory from Sections \ref{s:convexity} and \ref{s:MaxSlope} that the construction in Section \ref{s:GMM} indeed gives a weak solution of \eqref{eq:osmosis_scaled}. Finally, in Section \ref{s:permeability}, a model for aquaporins will be presented. It will be shown that a very small modification of the metric space is sufficient to obtain a gradient flow formulation for the resulting problem.

\section{Construction of a generalized minimizing movement}
\label{s:GMM}

In order to obtain a gradient flow solution for \eqref{eq:osmosis_scaled} under the assumption of radial symmetry, a metric space will be constructed to represent the domain and mass profile. It will be shown that this space, equipped with a weak topology, and a suitable functional fit the theory in \cite[\S2]{Ambrosio_Gigli_Savare}. More precisely, a so-called Generalized Minimizing Movement or GMM will be constructed. The idea of this concept is a (formal) generalization of the Euler backward approximation of a gradient flow in $\Reals^n$. 

The gradient flow of a function $\phi$ on $\Reals^n$ can be constructed by an Euler backward scheme. Time is discretized by setting $t_k=kh$ for some small $h>0$, and the equation
\begin{equation}\label{eq:euler_backward}
	\nabla\phi(x_{k+1})+\frac{x_{k+1}-x_k}{h}=0,
\end{equation}
is solved for every timestep. Setting $x(kh)=x_k$, this is an approximation of the gradient flow of $\phi$.

Note that \eqref{eq:euler_backward} is the optimality condition for the minization problem
\begin{equation}\label{eq:time_step_minimization}
	\min_{y\in\Reals^n}\Set{\phi(y)+\frac{1}{2h}|y-x|^2}.
\end{equation}
This last problem can easily be generalized to the setting of a metric space by replacing $|y-x|^2$ with $d^2(x, y)$. It is expected that this will, in some sense, give an approximation of a gradient flow. In general, \eqref{eq:time_step_minimization} might not have a unique solution, or no solution of all. Moreover, one must ask whether the limit for $h\downto 0$ exists, in what sense, and whether this limit is unique.

The concept of a (Generalized) Minimizing Movement is defined in \cite[Definition 2.0.6]{Ambrosio_Gigli_Savare}. The starting point is a partition of the time interval $(0, +\infty)$
\begin{equation}
	\mathcal{P}_{\boldsymbol{\tau}}
		:=\Set{0=t^0_{\boldsymbol{\tau}}<t^1_{\boldsymbol{\tau}}<\cdots<t^k_{\boldsymbol{\tau}}<\cdots}
\end{equation}
associated to a sequence $\boldsymbol{\tau}=\set{\tau_k}_{k\in\Naturals}$ of positive time steps with
\begin{equation}
	\lim_{k\to\infty}t^k_{\boldsymbol{\tau}}=\sum_{j=1}^\infty\tau_j=+\infty.
\end{equation}
and $|\boldsymbol{\tau}|:=\sup_k\set{\tau_k}<+\infty$. Additionally, the time intervals
\begin{equation}
	I^k_{\boldsymbol{\tau}}:=(t^{k-1}_{\boldsymbol{\tau}}, t^k_{\boldsymbol{\tau}}]
\end{equation}
are introduced.

Given a complete metric space $(\mathcal{S}, d)$ and a functional $\phi:\mathcal{S}\to(-\infty, +\infty]$, a discrete solution associated to the partition $\mathcal{P}_{\boldsymbol{\tau}}$ is a map $\overline{U}_{\boldsymbol{\tau}}:(0, +\infty)\to\mathcal{S}$ such that $\overline{U}_{\boldsymbol{\tau}}$ is constant on $I^k_{\boldsymbol{\tau}}$ for every $k\in\Naturals$, and $\overline{U}_{\boldsymbol{\tau}}(t^k_{\boldsymbol{\tau}})$ is a minimizer of the functional
\begin{equation}
	V\mapsto\phi_{\tau_k}(\overline{U}_{\boldsymbol{\tau}}(t^{k-1}_{\boldsymbol{\tau}}), V)
		:=\phi(V)+\frac{1}{2\tau_k}d^2(\overline{U}_{\boldsymbol{\tau}}(t^{k-1}_{\boldsymbol{\tau}}), V).
\end{equation}

A curve $[0, +\infty)\to\mathcal{S}$ is called a \emph{minimizing movement} for $\phi$ starting at $u_0\in\mathcal{S}$ if for every partition $\boldsymbol{\tau}$ with $|\boldsymbol{\tau}|$ sufficiently small there exists a discrete solution $\overline{U}_{\boldsymbol{\tau}}$ such that
\begin{equation}\label{eq:mm_def}
	\lim_{|\boldsymbol{\tau}|\downto 0}\phi(\overline{U}_{\boldsymbol{\tau}}(0))=\phi(u_0), \qquad
	\limsup_{|\boldsymbol{\tau}|\downto 0}d(\overline{U}_{\boldsymbol{\tau}}(0), u_0)<+\infty,
\end{equation}
and $\overline{U}_{\boldsymbol{\tau}}(t)\to u(t)$ for all $t\ge 0$ as $|\boldsymbol{\tau}|\downto 0$. In this definition, convergence of the discrete solutions with respect to the metric may be replaced by convergence with respect to a weaker topology. This allows for the use of a topology that may enjoy better compactness properties than the metric topology.

A \emph{generalized minimizing movement} is defined analogous to a minimizing movement, but with \eqref{eq:mm_def} required only along a sequence $\boldsymbol{\tau}_j$ of parititions. That is, a curve $[0, +\infty)\to\mathcal{S}$ is a generalized minimizing movement for $\phi$ starting at $u_0\in\mathcal{S}$ if there exists a sequence of partitions $\boldsymbol{\tau}_j$ with $\lim_{j\to\infty}|\boldsymbol{\tau}_j|=0$ and discrete solutions $\overline{U}_{\boldsymbol{\tau}_j}$ such that
\begin{equation}\label{eq:gmm_def}
	\lim_{j\to\infty}\phi(\overline{U}_{\boldsymbol{\tau_j}}(0))=\phi(u_0), \qquad
	\limsup_{j\to\infty}d(\overline{U}_{\boldsymbol{\tau_j}}(0), u_0)<+\infty,
\end{equation}
and $\overline{U}_{\boldsymbol{\tau}_j}(t)\to u(t)$ for all $t\ge 0$ as $j\to\infty$. Again, metric convergence may be replaced with convergence with respect to a weaker topology.

Examples of the construction of a GMM can be found in \cite{Jordan_Kinderlehrer_Otto}, \cite{Almgren_Taylor_Wang} and \cite{Luckhaus_Sturzenhecker}, where this approach is applied. In this section, a metric space $(\States, \varrho)$ and a functional $\Energy$ on $\States$ will be introduced. Moreover, it will be shown that GMM's exist for initial conditions that have finite $\Energy$-value by applying \cite[Proposition 2.2.3]{Ambrosio_Gigli_Savare}, which is proven in \S3 of the same book. This results in the following theorem.

\bigskip

\subsection{The space of balls}

By radial symmetry, the free domain will be a ball centered at the origin, which means it can be represented by a non-negative real number: the space of balls centered at zero is $\Cells:=[0, +\infty)$. The abbreviation $B_r$ for $B_r(0)$ will be used throughout the paper. Following \cite{Almgren_Taylor_Wang}, \cite{Luckhaus_Sturzenhecker}, \cite{Brakke} and many others, in order to obtain the mean curvature term the perimeter $\Perimeter(r)=n\omega_n r^{n-1}$ of $\partial B_r$ will be used.

Although metrics on $\Cells$ are trivial to construct and study, it is interesting to view them as the restriction of metrics on a suitable space of domains to the space of balls. 

A natural way to construct such a metric is described in \cite[\S2]{Grunewald_Kim}. One could consider all families $\set{E_t:t\in[0, 1]}$ of domains with smooth boundary, where $E_0$ and $E_1$ are prescribed, such that
\begin{equation}
	\bigcup_{t\in[0, 1]}\bdry E_t\times\set{t}
\end{equation}
is a smooth, $n$-dimensional manifold. If $v_t$ is the normal velocity of $\bdry E_t$, the integral
\begin{equation}\label{eq:domain_length}
	\int_0^1\left(\int_{\bdry E_t}v_t^2\ud{\Hausdorff^{n-1}}\right)^\frac{1}{2}\ud{t},
\end{equation}
is positive whenever $E_0\ne E_1$. One would like to define the distance between $E_0$ and $E_1$ to be the infimum of \eqref{eq:domain_length} over all such families $\set{E_t:t\in [0, 1]}$. Unfortunately, this infimum is in general not positive, even if $E_0$ and $E_1$ are different. However, if the $E_t$ are restricted to be balls centered at zero, the infimum can be computed explicitly:
\begin{equation}\label{eq:def_set_dist}
	\setdist(r_0, r_1)
		:=\int_{(r_0, r_1)}\sqrt{\Perimeter(\rho)}\ud{\rho}
		=\frac{2\sqrt{n\omega_n}}{n+1}\left|r_1^\frac{n+1}{2}-r_0^\frac{n+1}{2}\right|,
\end{equation}
where $E_0=B_{r_0}$ and $E(1)=B_{r_1}$. The fact that the above infimum is positive when restricted to $\Cells$ is the main reason to consider only the radially symmetric case.

In the general situation, the integral of a distance over the symmetric difference of two domains in \cite[\S2.6]{Almgren_Taylor_Wang} and \cite[(1)]{Jordan_Kinderlehrer_Otto} instead of $\setdist$:
\[
	\tilde\setdist^2(B_{r_0}, B_{r_1})
		=2\int_{B_{r_0}\triangle B_{r_1}}d(x, \bdry B_{r_0})\ud{x}.
\]
However, this is not a metric: it is not symmetric and does not satisfy the triangle inequality.

If it does not matter which one is used when constructing solutions, one would expect that, at an infinitesimal level, the two are the same. This is indeed the case: by evaluating $\tilde\setdist$ for balls,
\begin{equation}\begin{split}
  \tilde\setdist^2(r_0, r_1)
	&	=2\int_{B_{r_0}\triangle B_{r_1}}d(x, \bdry B_{r_0})\ud{x}
		=2n\omega_n\left(\frac{r_1^{n+1}}{n+1}+\frac{r_0^{n+1}}{n(n+1)}-\frac{r_1^n r_0}{n}\right),
\end{split}\end{equation}
so that,
\begin{equation}
  \left.\diff{t}\tilde\setdist(r_0, r_t)\right|_{t=0}
    =|r'_0|\sqrt{n\omega_n r_0^{n-1}}
    =\left(\int_{\partial B_{r_0}}|v|^2\ud{\Hausdorff^{n-1}}\right)^\frac{1}{2}
\end{equation}
since $v=r'$ for balls. Intuitively, this means that the two metrics have `the same infinitesimal structure'. What this means exactly, and what the implications for gradient flows are is being investigated by the author. Here, \eqref{eq:def_set_dist} will be used as the metric on $\Cells$.

Clearly, the map
\begin{equation}\label{eq:ball_isometry}
 	\iota_n:\Cells\to[0, \infty):r\mapsto\frac{2\sqrt{n\omega_n}}{n+1}r^\frac{n+1}{2}
\end{equation}
is an isometry, which will be used in computations. One can see immediately that $(\Cells, \setdist)$ is a complete metric space. Moreover, \eqref{eq:def_set_dist} generates the standard topology on $\Cells$, which means in particular that $\Perimeter$ is continuous with respect to $\setdist$. Finally, $r_k\to r$ is equivalent to $\Lebesgue^n(B_{r_k}\triangle B_r)\to 0$ as $k\to\infty$.

\bigskip

\subsection{Mass profiles}

The mass profile will be a radially symmetric nonnegative function with a fixed integral. Identifying a profile with a function $u$ on $(0, \infty)$, fixing the total mass is equivalent to requiring
\begin{equation}
  \int_0^\infty u(\rho)\Perimeter(\rho)\ud{\rho}=1
\end{equation}

Hence, an appropriate space for the solute profile is the space $\Profiles$ of radially symmetric probability density functions on $\Reals^n$, which is a metric subspace of $L^1(\Reals^n)$. Alternatively, it can be seen as a subspace of $L^1((0, \infty), \Perimeter)$, the space of integrable functions on $(0, \infty)$ with weight $\Perimeter$.

The topology on this space will be the weak topology inherited from $L^1(\Perimeter\Lebesgue^1)$, by definition, this means that $u_k\weaklyto u$ if and only if
\begin{equation}\label{eq:weak_l1_def}
	\int_0^\infty u_k(\rho) g(\rho)\Perimeter(\rho)\ud{\rho}\to\int_0^\infty u(\rho)g(\rho)\Perimeter(\rho)\ud{\rho}
	\qquad\forall g\in L^\infty((0, \infty)).
\end{equation}

The following compactness result is an easy consequence of the Dunford-Pettis theorem \cite[Theorem 1.38]{Ambrosio_Fusco_Pallara}.

\begin{thm}[Relative sequential compactness]\label{th:profile_compactness}
	Let $\set{u_k}_{k\in\Naturals}$ be a sequence in $\Profiles$ such that
	\begin{itemize}
	\item	for any $\eps>0$, there exists an $R>0$ such that $\int_R^\infty u_k(\rho)\Perimeter(\rho)\ud{\rho}<\eps$ for all $k\in\Naturals$,
	\item	there exists a nondecreasing function $g:[0, \infty)\to[0, \infty)$ with superlinear growth and a constant $C\in\Reals$ such that 
				\begin{equation}
					\int_0^\infty g(u_k)\Perimeter\ud{\Lebesgue^1}\le C.
				\end{equation}
				for all $k\in\Naturals$.
	\end{itemize}
	Then $\set{u_k}_{k\in\Naturals}$ has a subsequence $\set{u_{k(j)}}_{j\in\Naturals}$ converging weakly to $u\in\Profiles$.
\end{thm}
\begin{proof}
	By \cite[Proposition 1.27]{Ambrosio_Fusco_Pallara}, $\set{u_k}_{k\in\Naturals}$ is equi-integrable. Therefore, using the Dunford-Pettis theorem \cite[Theorem 1.38]{Ambrosio_Fusco_Pallara}, it has a subsequence $\set{u_{k_j}}_{j\in\Naturals}$ that converges weakly to some $u\in L^1((0, \infty), \Perimeter)$. Since $u_k(\rho)\ge 0$, it follows that $u(\rho)\ge 0$ for $\Lebesgue^1$-almost all $\rho>0$. Moreover, since the constant function is certainly in $L^\infty((0, \infty))$, 
\[
	\int_0^\infty u(\rho)\Perimeter(\rho)\ud{\rho}=\lim_{k\to\infty}\int_0^\infty u_k(\rho)\Perimeter(\rho)\ud{\rho}=1,
\]
which means that $u\in\Profiles$.
\end{proof}

Note that all that needed to be shown is that $\Profiles$ is closed as a subset of $L^1((0, \infty), \Perimeter)$.

The metric on $\Profiles$ set will be the Wasserstein distance, which is based on the optimal transportation problem. Given two mass profiles $u$ and $w$, one can consider all measurable maps $T:[0, \infty)\to [0, \infty)$ such that the profile $u$ turns into $w$ if all mass is transported according to the map $T$. Mathematically, this means that the \emph{push-forward} $T_\#u$ of $u$ under $T$, defined by
\begin{equation}
	\int_A u(\rho)\Perimeter(\rho)\ud{\rho}=\int_{T(A)} T_\#u(\rho)\Perimeter(\rho)\ud{\rho}
\end{equation}
for every Borel set $A\subset[0, \infty)$, must be equal to $v$. The Wasserstein metric is defined as
\begin{equation}
	W_2(u, v):=\left(\inf_{T:T_\#u=v}\int_0^\infty |x-T(x)|^2\ud{x}\right)^\frac{1}{2},
\end{equation}
which can be interpreted to be the minimum amount of work needed to transport mass from a profile $u$ to another profile $w$. For an extensive introduction of the optimal transportation problem and the Wasserstein metric, see \cite[\S6--7]{Ambrosio_Gigli_Savare}. The heat equation as a gradient flow with respect to the Wasserstein metric is treated as an example in this book.

The optimal transport problem in this setting is essentially the same as the one-dimensional optimal transportation problem. Therefore, the result of \cite[Theorem 6.0.2]{Ambrosio_Gigli_Savare} can be carried over to this setting. Writing
\begin{equation}\label{eq:cumulative_density}
	F_u(\rho):=\int_0^\rho u\Perimeter\ud{\Lebesgue^1}
\end{equation}
the optimal transport map  $t_u^w:[0, \infty)\to[0, \infty)$ and Wasserstein metric are given by
\begin{gather}
	t_u^w=F_w^{-1}\circ F_u \label{eq:rearrangement} \\
	W_2^2(u, w)
		=\int_0^\infty|\rho-F_w^{-1}\circ F_u(\rho)|^2 u(\rho)\Perimeter(\rho)\ud{\rho}
		=\int_{(0, 1)}\left|F_u^{-1}(\sigma)-F_w^{-1}(\sigma)\right|^2\ud{\sigma} \label{eq:one_dim_wasserstein}
\end{gather}
whenever $F_u$ and $F_w$ are invertible. If $F_u$ or $F_w$ is not invertible, the same identities hold if the inverse is replaced by the pseudo-inverse
\begin{equation}
	F_u^{-1}(\sigma):=\sup\Set{\rho\ge 0:F_u(\rho)\le\sigma},\qquad \sigma\in[0, 1].
\end{equation}
This observation will be useful later for studying convexity properties of the Wasserstein metric. 

In \cite[Proposition 7.1.3]{Ambrosio_Gigli_Savare}, it was shown that the Wasserstein metric is lower semicontinuous with respect to the topology of narrow convergence. Since the weak $L^1$ topology is stronger, it follows immediately that $W_2$ is also lower semicontinuous with respect to this topology.

Note that $\Profiles$, equipped with Wasserstein metric is \emph{not} a complete metric space, a sequence of profiles concentrating at the origin has no limit in $\Profiles$. Moreover, the weak $L^1$ topology is not necessarily weaker than the topology generated by the Wasserstein distance. It will turn out that this will not be a problem due to the choice of the functional, which will be defined below.

As in \cite{Jordan_Kinderlehrer_Otto}, the entropy functional will be used. Given a mass profile, its \emph{entropy} is defined as
\[
	\int_0^\infty u(\rho)\log u(\rho)\Perimeter(\rho)\ud{\rho},
\]
which is, up to a constant, Boltzmann's Entropy. Since the exact formula of the integrand is of little importance to the analysis, $z\log z$ will be replaced by a more general function.

\begin{defin}[Internal energy]\label{def:internal_energy}
	Let $\integrand:[0, \infty)\to\Reals$ be a function such that
	\begin{itemize}
	\item $f(0)=0$,
	\item $\lim_{z\downto 0}f(z)=0$
	\item $f$ is strictly convex,
	\item	$\lim_{z\to+\infty}\frac{\integrand(z)}{z}=+\infty$,
	\item	$f$ is continuously differentiable on $(0, \infty)$.
	\end{itemize}
	Then
	\begin{equation}
		\Entropy(u):=\int_0^\infty\integrand(u(\rho))\Perimeter(\rho)\ud{\rho}
	\end{equation}
	is called the \emph{internal energy} of $u$.
\end{defin}
\begin{rem}
	By continuity of $\integrand$ and superlinear growth, it follows that $\integrand$ has a global minimum, say 
	\begin{equation}\label{eq:integrand_lower_bound}
		\integrand(z)\ge-\integrand_0.
	\end{equation}
	Note that the minimum of $\integrand$ does not need to be negative, $\integrand$ may be increasing.
\end{rem}

Since the integrand is assumed to be convex, the internal energy defines a weakly lower semicontinuous functional on $\Profiles$. This has been proven in  a much more general setting in \cite[Theorem 2.34]{Ambrosio_Fusco_Pallara} by writing $\integrand$ as the supremum of countably many linear functions.

In what follows, the following auxiliary function will be used:
\begin{equation}
	\hat\integrand:[0, \infty)\to[0, \infty);z\mapsto
		\begin{cases}
			z\integrand'(z)-\integrand(z), &	\text{if $z>0$}, \\
			0, &															\text{if $z=0$}.
		\end{cases}
\end{equation}
Note that, since $\integrand$ is convex, $\hat\integrand$ is increasing, nonnegative and continuous. If $\integrand(z)=z\log z$, $\hat\integrand(z)=z$.

\bigskip

\subsection{The variational formulation}

The two spaces from the previous subsections can be combined into one space. This requires some caution, because the mass is supposed to stay \emph{inside} the free domain. Therefore, the following space will be used.
\begin{equation}
	\States:=\Set{(r, u)\in\Cells\times\Profiles:\int_r^\infty u\Perimeter\ud{\Lebesgue^1}=0}
\end{equation}
Note that the support of $u$ need not be the whole ball $B_r$.

The metric on this space will be
\begin{equation}
	\varrho((r, u), (s, w)):=\left(\setdist^2(r, s)+\frac{W_2^2(u, w)}{\diffusion}\right)^\frac{1}{2}.
\end{equation}
Moreover, since $\States$ is a subspace of a product, it inherits a topology from $\Cells$ and $\Profiles$: weak convergence in $\States$, denoted by $(r_k, u_k)\weaklyto(r, u)$ is equivalent to $r_k\to r$ and $u_k\weaklyto u$, where $u_k\weaklyto u$ denotes weak $L^1_n$ convergence, characterized by \eqref{eq:weak_l1_def}.

An obvious question is if the compactness properties of the spaces $\Cells$ and $\Profiles$ carry over to $\States$. Clearly, using Theorem \ref{th:profile_compactness} and the Heine-Borel theorem for $\Reals^1$, suitable sequences $\Cells\times\Profiles$ have convergent subsequences. The question is, however, if the condition that all mass should stay inside the varying domain can be carried over from a sequence in $\States$ to its limit, if it exists. The following theorem states that this can be done.

\begin{thm}\label{th:statespace_closed}
	$\States$ is closed as a subspace of $\Cells\times\Profiles$ in the weak topology.
\end{thm}
\begin{proof}
	Suppose that $(r_k, u_k)\weaklyto (r, u)$ in $\Cells\times\Profiles$ with $(r_k, u_k)\in\States$ for all $k\in\Naturals$. By the Dunford-Pettis theorem \cite[Theorem 1.38]{Ambrosio_Fusco_Pallara}, $\set{u_k}_{k\in\Naturals}$ is equi-integrable. Let then $\eps>0$, and pick $\delta>0$ be such that 
	\begin{equation}\label{eq:statespace_closed_no_accumulation}
	  \Lebesgue^n(A)<\delta\implies\int_A u_k(x)\ud{x}<\eps
	\end{equation}
	for all $k\in\Naturals$ and all Borel sets $A\subset\Reals^n$. Choose also $N\in\Naturals$ such that, $\omega_n(r_k^n-r^n)<\delta$ for all $k\ge N$. Then,
	\begin{equation}
		\int_r^\infty u_k(\rho)\Perimeter(\rho)\ud{\rho} 
			\le\int_{\max\set{r_k, r}}^\infty u_k(\rho)\Perimeter(\rho)\ud{\rho}
					{}+\int_r^{\max\set{r_k, r}} u_k(\rho)\Perimeter(\rho)\ud{\rho}<\eps,
	\end{equation}
	because the first integral is zero by definition of $\States$ and the second integral is smaller that $\eps$ by \eqref{eq:statespace_closed_no_accumulation} for all $k\ge N$. Taking limits for $k\to\infty$, it follows that
	\begin{equation}
	  \int_r^\infty u(\rho)\Perimeter(\rho)\ud{\rho}\le\eps,
	\end{equation}
	for any $\eps>0$. Then the integral must be zero, which means that $(r, u)\in\States$.
\end{proof}

\bigskip

The functional that will be used is a combination of the perimeter and the internal energy:
\begin{equation}
	\Energy(r, u):=\Perimeter(r)+\Entropy(u).
\end{equation}
In order to obtain coercivity, an assumption on the integrand $\integrand$ in relation to the dimension is made:
\begin{equation}\label{eq:coercivity_assumption}
	\lim_{z\downto 0}z^{-\frac{1}{n}}\integrand(z)=0.
\end{equation}
This condition can be interpreted as a coercivity condition. Without it, very large balls with a constant mass profile have very low $\Energy$-values. This would, in the end, result in the domain growing indefinitely large. A nice example is the case where the dimension $n$ is $1$ and $\integrand(z)=z\log z$. Note that, in particular, the choice $\integrand(z)=z\log z$ \emph{does} satisfy the assumption, provided that $n>1$. In order to ensure coercivity, it will be assumed that $n>1$ and \eqref{eq:coercivity_assumption} holds in the remainder.

It turns out that, once the assumption is made, the sublevels of $\Energy$ are complete with respect to $\varrho$ and the weak topology is the same as the metric topology on sublevels. Before proving this, some basic properties of $\Energy$ will be shown.

\begin{lem}\label{th:energy_coercivity}
	Given $r\in\Cells$, $u\mapsto\Energy(r, u)$ takes its unique minimal value 
	\begin{equation}\label{eq:energy_constant}
		\Energy_o(r):=n\omega_n r^{n-1}+\omega_n r^n\integrand\left(\frac{1}{\omega_n r^n}\right)
	\end{equation}
	at $u=\frac{1}{\omega_n r^n}\charf{(0, r)}$. The function $\Energy_o$ is coercive, in the sense that
	\begin{itemize}
	\item	$\lim_{r\to+\infty}\Energy_o(r)=+\infty$,
	\item $\lim_{r\downto 0}\Energy_o(r)=+\infty$,
	\item $r\mapsto\Energy_o(r)$ has a unique global minimum.
	\end{itemize}
\end{lem}
\begin{proof}
	Since $\integrand$ is convex, Jensen's inequality \cite[Theorem 2.4.19]{Federer} implies that $u=\frac{1}{\omega_n r^n}\charf{(0, r)}$ minimizes $\Energy(r, u)$ for given $r>0$.	Superlinear growth of $\integrand$ implies that $\lim_{r\downto 0}\Energy_o(r)=+\infty$. Moreover, $\lim_{r\to\infty}\Energy_o(r)=+\infty$ using \eqref{eq:coercivity_assumption}. Since
	\begin{equation}
		\diff{z}\left(z\integrand\left(\frac{1}{z}\right)\right)
			=-\hat\integrand\left(\frac{1}{z}\right)\ge 0,
	\end{equation}
	is strictly increasing, the second term in \eqref{eq:energy_constant} is convex. Since the first term is obviously also convex, it follows that $r\mapsto\Energy_o(r)$ has a unique global minimum.
\end{proof}

Using this coercivity lemma, the following two properties of sublevels of $\Energy$ can be shown.

\begin{thm}\label{th:sublevel_compact}
	For any $M\in\Reals$, $\Energy^{-1}(-\infty, M]:=\set{(r, u)\in\States:\Energy(r, u)\le M}$ is weakly sequentially compact.
\end{thm}
\begin{proof}
	Let $\set{(r_k, u_k)}_{k\in\Naturals}$ be a sequence in $\Energy^{-1}(-\infty, M]$. By Lemma \ref{th:energy_coercivity}, $\set{r_k}$ is uniformly bounded away from $0$ and $+\infty$, say $c<r_k<C$, which means it must have a convergent subsequence, say $r_k\to r>0$. Taking the corresponding sequence in $u_k$, it follows that
	\begin{equation}
		\int_0^\infty\max\set{\integrand(u(\rho)), 0}\Perimeter(\rho)\ud{\rho}\le M-\omega_n C^{n-1}+\omega_n C^n\integrand_0
	\end{equation}
	along this subsequence. Since also $u_k(\rho)=0$ for any $\rho\ge C$, Theorem \ref{th:profile_compactness} with $g=\max\set{\integrand, 0}$ implies that there must be a further subsequence converging weakly to some $u\in\Profiles$. Taking the corresponding subsequence for $r_k$, this results in a subsequence $\set{(r_{k(j)}, u_{k(j)})}_{j\in\Naturals}$ such that $(r_{k(j)}, u_{k(j)})\to(r, u)$ for some $(r, u)\in\Cells\times\Profiles$. By \ref{th:statespace_closed}, $(r, u)\in\States$. Moreover, by continuity of $P_n$ and weak lower semicontinuity of $\Entropy$, $\Energy(r, u)\le\liminf_{k\to\infty}\Energy(r_k, u_k)\le M$.
\end{proof}

\begin{lem}\label{th:sublevel_topology}
	Let $\set{(r_k, u_k)}_{k\in\Naturals}$ be a sequence in $\Energy^{-1}(-\infty, M]$. Then $(r_k, u_k)\weaklyto (r, u)$ if and only if $(r_k, u_k)\to(r, u)$ with respect to $\varrho$.
\end{lem}
\begin{proof}
	Suppose first that $(r_k, u_k)\weaklyto (r, u)$. Then, as above, $u_k(\rho)=0$ for $\rho>C$. Therefore, the corresponding measures on $\Reals^n$ have uniformly integrable second moments. By definition, weak convergence in $L^1$ of $u_k$ implies narrow convergence of the measures $\mu_k$ to $\mu$. Applying \cite[Proposition 7.1.5]{Ambrosio_Gigli_Savare}, it follows that $u_k$ converges to $u$ with respect to the Wasserstein metric. Then $(r_k, u_k)\to(r, u)$ with respect to $\varrho$.
	
	Conversely, assume $(r_k, u_k)\to(r, u)$. Then boundedness of $r_k$, together with Theorem \ref{th:profile_compactness} implies that $(r_{k(j)}, u_{k(j)})$ must have weak limit point $(r', u')$. Clearly, this implies $r'=r$. By definition of \weakstar\ convergence of measures and weak convergence in $L^1$, the measures $\mu_{k_j}$ corresponding to $u_{k_j}$ must \weakstar-converge to $\mu'$, the measure corresponding to $u'$. On the other hand, since $u_k\to u$, $\mu_{k_j}\weakstarto\mu$, the measure corresponding to $u$ as well. This implies $\mu'=\mu$, which means that $u'=u$ almost everywhere.
\end{proof}

The above analysis shows that the metric space 
\begin{gather}
		\States=\Set{(r, u)\in\Cells\times\Profiles:\int_r^\infty u\Perimeter\ud{\Lebesgue^1}=0}, \\
		\varrho((r, u), (s, w)),=\left(\setdist^2(r, s)+\frac{W_2^2(u, w)}{\diffusion}\right)^\frac{1}{2}
\end{gather}
and the functional
\begin{equation}
	\Energy(r, u)=\Perimeter(r)+\Entropy(u)=n\omega_n r^{n-1}+\int_0^\infty f(u(\rho))\Perimeter(\rho)\ud{\rho}
\end{equation}
satisfy the requirements of \cite[\S2.1]{Ambrosio_Gigli_Savare} \emph{if} the problem of finding a GMM can be restricted to a sublevel of $\Energy$. Lower semicontinuity of $\varrho$ and $\Energy$, coercivity of $\Energy$ hold on the whole of $\States$. Since on sublevels of $\Energy$, the weak and metric topology coincide, strong compactness follows from weak sequential compactness. In particular, sublevels of $\Energy$ are complete.

By requiring that $\Energy$ is finite for the initial condition, it is no problem to restrict the construction to a sublevel of $\Energy$. Therefore, \cite[Proposition 2.2.3]{Ambrosio_Gigli_Savare} applies, which finishes the proof of 

\begin{thm}\label{th:GMM_existence}
	Let $(r_0, u_0)\in\States$, such that $\Energy(r_0, u_0)$ is finite. Then a GMM for $\Energy$ starting from $(r_0, u_0)$ exists.
\end{thm}

\section{Interpolation, Convexity and Uniqueness}
\label{s:convexity}

In this section, the convexity properties of the problem are studied. The goal is to prove that \cite[Theorems 2.4.15 and 4.0.4]{Ambrosio_Gigli_Savare} can be applied. This implies a number of properties of the GMM constructed in the previous section, and eliminates the finiteness condition from Theorem \ref{th:GMM_existence}.

The main assumption is convexity of the minimization problem solved in every time step of the discretized problem. 
\begin{assumption}\label{as:convexity}
	For every $w$, $v_0$, $v_1$ in $(\states, d)$, there exists an interpolating curve $\gamma:[0, 1]\to\states$ such that $\gamma(0)=v_0$, $\gamma(1)=v_1$, and the map
	\begin{equation}\label{eq:moreau_yosida}
		v\mapsto\phi(v)+\frac{1}{2h}d^2(w, v)
	\end{equation}
	is $(\frac{1}{h}+\lambda)$-convex along $\gamma$ for every $0<h<\frac{1}{\lambda^-}$.
\end{assumption}
\begin{rem}
	This assumption implies that for $h<\frac{1}{\lambda^-}$, any two points are connected by a curve along which \eqref{eq:moreau_yosida} is strictly convex. In particular, this means that minimizers are unique. In turn, this implies that for $|\mathbf{\tau}|<\frac{1}{\lambda^-}$, discrete solutions associated to $\mathbf{\tau}$ are unique.
\end{rem}

The definition of $\lambda$-convexity along a curve will be recalled below. As is the case with ordinary convexity, this assumption is only useful if $v_0$ and $v_1$ have finite $\phi$-value, but $w$, $v_0$ and $v_1$ can be restricted to an even smaller set, see \cite[Assumption 4.0.1]{Ambrosio_Gigli_Savare} for the details. 

Before showing that $(\States, \varrho)$ and $\Energy$ satisfy this assumption, some finer properties of $(\States, \varrho)$ will be analyzed further.

\bigskip

\subsection{Absolute continuity and the metric derivative}

In order to study interpolating curves, and geodesics of $\varrho$ in particular, it is convenient to first study the notion of absolute continuity and metric derivative, as introduced in \cite[\S1.1]{Ambrosio_Gigli_Savare}. Summarizing, a curve $\gamma$ parametrized on $[0, 1]$ is absolutely continuous with respect to a metric $d$ if there exists a function $g\in L^1([0, 1])$ such that
\begin{equation}\label{eq:absolute_continuity_def}
	d(\gamma(s), \gamma(t))\le\int_s^t g(\tau)\ud{\tau}
\end{equation}
for all $0\le s<t\le 1$. If $\gamma$ is absolutely continuous, the limit
\begin{equation}
	|\gamma'|(\tau)=\lim_{h\to 0}\frac{d(\gamma(\tau+h), \gamma(\tau))}{|h|},
\end{equation}
called the \emph{metric derivative} of $\gamma$, exists for almost every $\tau$. Moreover, it is the smallest $g$ that satisfies \eqref{eq:absolute_continuity_def}. Note that existence of the limit does \emph{not} guarantee absolute continuity of the curve. In this section, the notion of absolute continuity will first be analyzed for $\setdist$ and $W_2$ first. 

\bigskip

First, absolute continuity in $\Cells$ is studied. Using the isometry \eqref{eq:ball_isometry}, it is easy to connect absolute continuity with respect to $\setdist$ to absolute continuity with respect to the Euclidean metric. If $\tau\mapsto r(\tau)$ is absolutely continuous with respect to $\setdist$,
\begin{equation}
	\setdist(r(s), r(t))\le\int_s^t |r'|_\setdist(\tau)\ud{\tau},
\end{equation}
where $|r'|_\setdist$ is the metric derivative of the curve with respect to $\setdist$. By applying the isometry $\iota_n$, the curve $\tau\mapsto\rho(\tau):=\iota_n(r(\tau))$ is absolutely continuous in $[0, \infty)$. It follows that $\rho(\tau)$ is differentiable almost everywhere, and that the metric derivative with respect to the Euclidean metric is just the absolute value $|\diff[\rho(\tau)]{\tau}|$ of the ordinary derivative $\diff[\rho(\tau)]{\tau}$ for almost every $\tau$. Since $\iota_n$ is an isometry, this absolute value must also be equal to $|r'|(\tau)$ for almost every $\tau$. An easy computation now shows that $r(\tau)$ must be differentiable, and
\begin{equation}\label{eq:set_dist_derivative}
	|r'|_\setdist(\tau)=\sqrt{\Perimeter(r(\tau))}\left|\diff[r(\tau)]{\tau}\right|.
\end{equation}
It follows that absolute continuity of $r$ is absolute continuity of $r$ as an $\Reals$-valued function, together with integrability of $\sqrt{\Perimeter(r(\tau))}r'(\tau)$.

There is also a useful characterization of absolute continuity involving a weak formulation. Clearly, if $r_t$ is smooth,
\begin{equation}
	\diff{\tau}\int_0^{r(\tau)}\psi(\rho, \tau)\Perimeter(\rho)\ud{\rho}
		=\int_0^{r(\tau)}\pdiff[\psi(\rho, \tau)]{\tau}\Perimeter(\rho)\ud{\rho}
			+\diff[r(\tau)]{\tau}\psi(r(\tau), \tau))\Perimeter(r(\tau)).
\end{equation}
for any smooth, radially symmetric test function $\psi$. As it turns out, this equation is sufficient to characterize absolute continuity of $r(\tau)$, and provides an expression for the metric derivative.

\begin{lem}\label{th:ball_continuity}
	Let $\tau\mapsto r(\tau)$ be a curve in $\Cells$. Then $r(\tau)$ is absolutely continuous if and only if it is continuous and there exists a function $g$ such that $g(\tau)\sqrt{\Perimeter(r(\tau))}\in L^1_\loc((0, \infty))$, and
	\begin{equation}\label{eq:ball_continuity}
		\int_0^\infty\int_0^{r(\tau)}\pdiff[\psi(\rho, \tau)]{\tau}\Perimeter(\rho)\ud{\rho}\ud{\tau}
			=-\int_0^\infty g(\tau)\psi(r(\tau), \tau)\Perimeter(r(\tau))\ud{\tau}
	\end{equation}
	for all $\psi\in C^\infty_c([0, \infty)\times(0, \infty))$. In this case, $|r'|_\setdist(\tau)=|g(\tau)|\sqrt{\Perimeter(r(t\tau))}$ for almost every $t>0$.
\end{lem}
\begin{proof}
	Suppose first that $r(\tau)$ is absolutely continuous. Then $|r'|_\setdist(\tau)=|\diff[r(\tau)]{\tau}|\sqrt{\Perimeter(r(\tau))}\in L^1_\loc((0, \infty))$ and,
	for any $\psi\in C^\infty_c([0, \infty)\times(0, \infty))$,
	\begin{equation}\begin{split}
		\int_0^\infty\int_0^{r(\tau)}\pdiff[\psi(\rho, \tau)]{\tau}\Perimeter(\rho)\ud{\rho}\ud{\tau}
		&	=\lim_{h\downto 0}\int_0^\infty\int_0^{r(\tau)}
					\frac{\psi(\rho, \tau)-\psi(\rho, \tau-h)}{h}\Perimeter(\rho)\ud{\rho}\ud{\tau} \\
		&	=\lim_{h\downto 0}\int_0^\infty\frac{1}{h}\left(\int_0^{r(\tau)}\psi(\rho, \tau)\Perimeter(\rho)\ud{\rho}\right. \\
		&	\hspace{3cm}	\left.{}-\int_0^{r(\tau+h)}\psi(\rho, \tau)\Perimeter(\rho)\ud{\rho}\right)\ud{\tau} \\
		&	=-\int_0^{r(\tau)}\diff[r(\tau)]{\tau}\psi(r(\tau), \tau)\Perimeter(r(\tau))\ud{\tau},
	\end{split}\end{equation}
	if $h>0$ is small enough. 
	
	For the opposite implication, assume that $r(\tau)$ is continuous and $g$ satisfies the equation. One would like to substitute $\psi(\rho, \tau)=\chi_{[t_0, t_1]}\phi(\rho)$ with $\phi\in C^\infty_c([0, \infty))$ into \eqref{eq:ball_continuity}, but this function is discontinuous. Let then $\zeta_k:(0, \infty)\to[0, 1]$ be a sequence of smooth functions such that $\zeta_k\to\chi_{[t_0, t_1]}$ pointwise. Substituting $\phi(\rho)\zeta_k(\tau)$ for $\psi$ in \eqref{eq:ball_continuity}, one obtains
	\begin{equation}\begin{split}
		\int_{t_0}^{t_1}\phi(r(\tau))g(\tau)\Perimeter(r(\tau))\ud{\tau}
		&	=\lim_{k\to\infty}\int_0^\infty g(\tau)\phi(r(\tau))\zeta_k(\tau)\Perimeter(r(\tau))\ud{\tau} \\
		&	=\lim_{k\to\infty}\int_0^\infty\zeta_k'(\tau)\int_0^{r(\tau)}\phi(\rho)\Perimeter(\rho)\ud{\rho}\ud{\tau} \\
		&	=\lim_{k\to\infty}\int_0^\infty\zeta_k(\tau)
						\diff{\tau}\left(\int_0^{r(\tau)}\phi(\rho)\Perimeter(\rho)\ud{\rho}\right)\ud{\tau} \\
		&	=\int_{r(t_0)}^{r(t_1)}\phi(\rho)\Perimeter(\rho)\ud{\rho}
	\end{split}\end{equation}
	if $\phi\in C^\infty_c([0, \infty))$ for all $t_0<t_1$. Using another approximation argument, this also holds for $\phi(\rho)=\Perimeter(\rho)^{-\frac{1}{2}}$, which implies that 
	\begin{equation}
		\setdist(r(t_0), r(t_1))
			=\left|\int_{r(t_0)}^{r(t_1)}\sqrt{\Perimeter(\rho)}\ud{\rho}\right|
			\le\left|\int_{t_0}^{t_1}g(\tau)\sqrt{\Perimeter(r(\tau))}\ud{\tau}\right|.
	\end{equation}
	Since $t_0$ and $t_1$ are arbitrary, this implies that $r(\tau)$ is absolutely continuous with metric derivative $|r'|_\setdist(\tau)=|g(\tau)|\sqrt{\Perimeter(r(\tau))}$.
\end{proof}

\bigskip

Next, absolute continuity in $\Profiles$ is considered. For this, it is helpful to remember that $\Profiles$ can be considered as a subspace of the space of probability measures on $\Reals^n$. A study of absolute continuity of curves in the space of measures can be found in \cite[\S8.3]{Ambrosio_Gigli_Savare}. The main theorem \cite[Theorem 8.3.1]{Ambrosio_Gigli_Savare} from this section states that, essentially, the metric derivative of a curve of measures $\tau\mapsto\mu(\tau)$ can be found by solving the continuity equation
\begin{equation}\label{eq:continuity_equation}
	\pdiff[\mu(\tau)]{\tau}+\Div(\mathbf{v}(\tau)\mu(\tau))=0
\end{equation}
in distributional sense. It is shown that, if $\tau\mapsto\mu(\tau)$ is absolutely continuous, the metric derivative at any time $t_0$ is the minimal $L^2(\mu(\tau))$-norm of solutions $\mathbf{v}(\tau)$ of \eqref{eq:continuity_equation}. Note that $\mathbf{v}(\tau)$ can be thought of as the velocity of the mass at a certain position and time. In particular, $\|v(\tau)\|^2_{L^2(\mu(\tau))}$ can be thought of as the kinetic energy.

In order to specialize \cite[Theorem 8.3.1]{Ambrosio_Gigli_Savare} to $\Profiles$, a particular space of test functions is needed: the space of functions $\psi:[0, \infty)\times (0, \infty)$ such that $\phi:\Reals^n\times(0, \infty)$ defined by $\phi(x, \tau)=\psi(|x|, \tau)$ is $C^\infty$ and compactly supported will be denoted by $C^\infty_{c, r}([0, \infty), (0, \infty))$.

\begin{thm}\label{th:continuity_equation}
	Let $\tau\mapsto u(\tau)$ be a weakly continuous curve in $\Profiles$. If $\tau\mapsto u(\tau)$ is absolutely continuous, there exists a Borel function $v:(\rho, \tau)\mapsto v_\tau(\rho)$ such that $v_\tau\in L^2((0, \infty), u(\tau)\Perimeter)$ with $\|v_\tau\|_{L^2((0, \infty), u(\tau)\Perimeter)}\le|u'|(\tau)$ for almost every $\tau$ and
	\begin{equation}\label{eq:rad_symm_cont}
		\int_0^\infty\int_0^\infty\left(\pdiff[\psi(\rho, \tau)]{\tau}
			{}+\pdiff[\psi(\rho, \tau)]{\rho}v_\tau(\rho)\right)u(\rho, \tau)\Perimeter(\rho)\ud{\rho}\ud{\tau}=0
	\end{equation}
	for all $\psi\in C^\infty_{c, r}([0, \infty), (0, \infty))$. Conversely, if there exists $v$ with $v_\tau\in L^2((0, \infty), u(\tau)\Perimeter)$ such that \eqref{eq:rad_symm_cont} holds for every $\psi\in C^\infty_{c, r}([0, \infty), (0, \infty))$, then $\tau\mapsto u(\tau)$ is absolutely continuous and $|u'|(\tau)\le\|v_\tau\|_{L^2((0, \infty), u(\tau)\Perimeter)}$ for almost every $\tau>0$.
\end{thm}
\begin{proof}
	Suppose first that $\tau\mapsto u(\tau)$ is absolutely continuous. By \cite[Theorem 8.3.1]{Ambrosio_Gigli_Savare}, there exists a Borel vector field $\mathbf{v}:(x, \tau)\mapsto\mathbf{v}_\tau(x)$ such that $\mathbf{v}_\tau\in L^2(\Reals^n, u)$ with $\|\mathbf{v}_\tau\|_{L^2(\Reals^n, u)}\le|u'|(\tau)$ for almost every $\tau\in (0, \infty)$. It follows that for such a vector field, \cite[(8.3.8)]{Ambrosio_Gigli_Savare} holds with $\phi$ restricted to radially symmetric, smooth, compactly supported functions on $\Reals^n\times (0, \infty)$. Under this restriction, \cite[(8.3.8)]{Ambrosio_Gigli_Savare} becomes invariant under rotations, that is, if $\tilde{\mathbf{v}}_\tau(x)=\mathbf{v}(Rx)$ with $R$ a rotation in $\Reals^n$, $\tilde{\mathbf{v}}$ also satisfies \cite[(8.3.8)]{Ambrosio_Gigli_Savare} for radially symmetric $\phi$. Hence, it may be assumed that $\mathbf{v}_\tau$ is radially symmetric for almost every $\tau>0$. Let $\psi\in C^\infty_{c, r}([0, \infty), (0, \infty))$ be given, and set $\phi(x, \tau):=\psi(|x|, \tau)$. Then, using \cite[(8.3.8)]{Ambrosio_Gigli_Savare},
	\begin{equation}\begin{split}
		0 &	=\int_0^\infty\int_{\Reals^n}\left(\pdiff[\phi(x, \tau)]{\tau}
					{}+\langle\mathbf{v}_\tau(x), \nabla_x\phi(x, \tau)\rangle\right)u(x)\ud{x}\ud{\tau} \\
		&		=\int_0^\infty\int_{\Reals^n}\left(\pdiff[\psi(|x|, \tau)]{\tau}
					{}+\left\langle\mathbf{v}_\tau(x), \frac{x}{|x|}\right\rangle\psi'(|x|, \tau)\right)u(x)\ud{x}\ud{\tau} \\
		&		=\int_0^\infty\int_0^\infty\left(\pdiff[\psi(\rho, \tau)]{\tau}
					{}+\langle\mathbf{v}_\tau(\rho e_1), e_1\rangle\psi'(\rho, \tau)\right)u(\rho)\Perimeter(\rho)\ud{\rho}\ud{\tau},
	\end{split}\end{equation}
	which means that $v_\tau(\rho):=\langle\mathbf{v}_\tau(\rho e_1), e_1\rangle$ solves \eqref{eq:rad_symm_cont}. Moreover, by construction,
	\begin{equation}
		\|v_\tau\|_{L^2((0, \infty), u(\tau)\Perimeter)}
			\le\|\mathbf{v}_\tau\|_{L^2(\Reals^n, u)}
			\le|u'(\tau)|
	\end{equation}
	for almost every $\tau>0$.
	
	For the converse implication, assume that $v_\tau$ satisfies \eqref{eq:rad_symm_cont} for every $\psi\in C^\infty_{c, r}([0, \infty), (0, \infty))$. Setting $\mathbf{v}_\tau(x):=v_\tau(|x|)\frac{x}{|x|}$. It is easily checked that $\mathbf{v}_\tau$ satisfies \cite[(8.3.8)]{Ambrosio_Gigli_Savare} for all radially symmetric test functions $\phi$. If $\phi\in C^\infty_c(\Reals^n\times(0, \infty))$ is a general test function, the function $\psi$, defined by
	\begin{equation}
		\psi(\rho, \tau):=\frac{1}{\Perimeter(\rho)}\int_{\bdry B_\rho}\phi(x, \tau)\ud{\Hausdorff^{n-1}x}
	\end{equation}
	is in $C^\infty_{c, r}([0, \infty), (0, \infty))$. By definition of $\mathbf{v}_\tau$ and $\psi$, it follows that
	\begin{multline}
		\int_0^\infty\int_{\Reals^n}\left(\pdiff[\phi(x, \tau)]{\tau}
					{}+\langle\mathbf{v}_\tau(x), \nabla_x\phi(x, \tau)\rangle\right)u(x)\ud{x}\ud{\tau} \\
			=\int_0^\infty\int_0^\infty\left(\pdiff[\psi(x, \tau)]{\tau}
					{}+\pdiff[\psi(\rho, \tau)]{\rho}v_\tau(\rho)\right)u(\rho)\Perimeter(\rho)\ud{\rho}\ud{\tau}
			=0,
	\end{multline}
	which means that $\mathbf{v}$ satisfies \cite[(8.3.8)]{Ambrosio_Gigli_Savare} for any test function. Hence, $\tau\mapsto u(\tau)$ is absolutely continuous, and $|u'(\tau)|\le\|\mathbf{v}_\tau\|_{L^2(\Reals^n, u)}=\|v_\tau\|_{L^2((0, \infty), u(\tau)\Perimeter)}$.
\end{proof}
\begin{rem}\label{rem:minimal_norm_solution}
	If follows in particular that, if $\tau\mapsto u(\tau)$ is absolutely continuous, there exists $v$ satisfying \eqref{eq:continuity_equation} such that $\|v_\tau\|_{L^2((0, \infty), u\Perimeter)}=|u'|(\tau)$ for almost all $\tau$. By \cite[Proposition 8.4.5]{Ambrosio_Fusco_Pallara}, this $v_\tau$ is uniquely determined for almost every $\tau$.
\end{rem}

Note the similarity between Lemma \ref{th:ball_continuity} and the theorem cited above: both relate absolute continuity and the metric derivative to the solvability and a norm of solutions of a certain weakly formulated equation. Both equations will prove useful later.

Intuitively, the optimal transport map and the solution of the continuity equation are related: the former is displacement of mass, the latter is velocity of mass. This intuition can be made precise using \cite[Propisition 8.4.6]{Ambrosio_Gigli_Savare}, where it is shown that if $\tau\mapsto u(\tau)$ is absolutely continuous,
\begin{equation}\label{eq:optimal_map_velocity}
	\lim_{h\to 0}\frac{t_{u(\tau)}^{u(\tau+h)}-\id}{h}=v_\tau
\end{equation}
for almost every $\tau$, where $v_\tau$ is the solution of \eqref{eq:continuity_equation} such that $\|v_\tau\|_{L^2((0, \infty), u(\tau)\Perimeter)}=|u'|(\tau)$ for almost every $\tau$.

\bigskip

Absolute continuity in $\States$ can now be characterized in terms of absolute continuity in $\Cells$ and $\Profiles$. That is,

\begin{lem}\label{th:states_abs_cont}
	A curve $\tau\mapsto(r(\tau), u(\tau))$ in $\Cells\times\Profiles$ is absolutely continuous if and only if the curves $\tau\mapsto r(\tau)$ and $\tau\mapsto u(\tau)$ are absolutely continuous in $\Cells$ and $\Profiles$, respectively. In this case,
	\begin{equation}
		|(r, u)'|(\tau)=\left(|r'|^2_\setdist(\tau)+\frac{|u'|_{W_2}^2(\tau)}{\diffusion}\right)^\frac{1}{2}
	\end{equation}
	for almost all $\tau$.
\end{lem}
\begin{proof}
	Let $\tau\mapsto (r(\tau), u(\tau))$ be a curve in $\Cells\times\Profiles$. By definition, 
	\begin{gather}
		\setdist(r(s), r(t))\le\varrho((r(s), u(s)), (r(t), u(t))), \\
		W_2(u(s), u(t))\le\sqrt{\diffusion}\varrho((r(s), u(s)), (r(t), u(t))),
	\end{gather}
	which means that absolute continuity of $\tau\mapsto (r(\tau), u(\tau))$ implies absolute continuity of $\tau\mapsto r(\tau)$ and $\tau\mapsto u(\tau)$. 
	
	Conversely, if both $\tau\mapsto r(\tau)$ and $\tau\mapsto u(\tau)$ are absolutely continuous,
	\begin{equation}\begin{split}
		\varrho((r(s), u(s)), (r(t), u(t)))
		&	=\left(\setdist^2(r(s), r(t))+\frac{W_2^2(u(s), u(t))}{\diffusion}\right)^\frac{1}{2} \\
		&	\le\left(\left(\int_s^t|r'|_\setdist(\tau)\ud{\tau}\right)^2
				{}+\left(\int_s^t\frac{|u'|_{W_2}(\tau)}{\sqrt{\diffusion}}\ud{\tau}\right)^2\right)^\frac{1}{2} \\
		&	\le 2\int_s^t\left(|r'|_\setdist^2(\tau)+\frac{|u'|_{W_2}^2(\tau)}{\diffusion}\right)^\frac{1}{2}\ud{\tau}
	\end{split}\end{equation}
	which means that $\tau\mapsto(r(\tau), u(\tau))$ is absolutely continous.
	
	Finally, if $\tau\mapsto(r(\tau), u(\tau))$ is absolutely continous,
	\begin{equation}\begin{split}
		|(r, u)'|(\tau)
		&	=\lim_{h\to 0}\frac{\varrho((r(\tau+h), u(\tau+h)), (r(\tau), u(\tau)))}{|h|} \\
		&	=\left(\left(\lim_{h\to 0}\frac{\setdist(r(\tau+h), r(\tau))}{|h|}\right)^2
					+\frac{1}{\diffusion}\left(\lim_{h\to 0}\frac{W_2(u(\tau+h), u(\tau))}{|h|}\right)^2\right)^\frac{1}{2} \\
		&	=\left(|r'|_\setdist^2(\tau)+\frac{|u'|_{W_2}(\tau)}{\diffusion}\right)^\frac{1}{2}
	\end{split}\end{equation}
	where, for almost every $\tau$, existence of all limits is guaranteed by absolute continuity.
\end{proof}

Together with the results about absolute continuity in $\Cells$ and $\Profiles$, this lemma is the main tool to study absolutely continuous curves in $\States$.

\bigskip

\subsection{Constant speed geodesics}

In a metric space $\mathcal{S}$, a constant speed geodesic is by defintion a curve $\gamma:[0, 1]\to\mathcal{S}$ satisfying
\begin{equation}\label{eq:geodesic_def}
	d(\gamma(s), \gamma(t))=(t-s)d(\gamma(0), \gamma(1))
\end{equation}
for any $0\le s\le t\le 1$. Note that this is different from geodesics on a Riemannian manifold: geodesics are parametrized by constant velocity on the unit interval, instead of parametrized by length. By the triangle inequality, it is sufficient to show only 
\[
	d(\gamma(s), \gamma(t))\le(t-s)d(\gamma(0), \gamma(1)).
\]
It is also clear that a constant speed geodesic is absolutely continuous, and the metric derivative is equal to $d(\gamma(0), \gamma(1))$ almost everywhere. The converse is also true: if $\gamma$ is absolutely continuous, and 
\[
	|\gamma'|_d(t)\le d(\gamma(0), \gamma(1)),
\]
then $\gamma$ is a constant speed geodesic. From this observation, it follows that $\tau\mapsto(r(\tau), u(\tau))$ is a constant speed geodesic if and only if the maps $\tau\mapsto r(\tau)$ and $\tau\mapsto u(\tau)$ are.

\begin{lem}\label{th:states_geodesics}
	A curve $\tau\mapsto\mapsto(r(\tau), u(\tau))$ in $\Cells\times\Profiles$ is a constant speed geodesic if and only if $\tau\mapsto r(\tau)$ and $\tau\mapsto u(\tau)$ are.
\end{lem}
\begin{proof}
	Suppose first that $\tau\mapsto r(\tau)$ and $\tau\mapsto u(\tau)$ are constant speed geodesics. Then, using Lemma \ref{th:states_abs_cont}, $\tau\mapsto(r(\tau), u(\tau))$ is absolutely continuous, and
	\begin{equation}\begin{split}
		|(r, u)'|^2(\tau)
		&	=|r'|_\setdist^2(\tau)+\frac{|u'|_{W_2}^2(\tau)}{\diffusion}
			=\setdist^2(r(0), r(1))+\frac{W_2^2(u(0), u(1))}{\diffusion} \\
		&	=\varrho^2((r(0), u(0)), (r(1), u(1))).
	\end{split}\end{equation}
	
	Conversely, suppose that $\tau\mapsto (r(\tau), u(\tau))$ is a constant speed geodesic. Using Lemma \ref{th:states_abs_cont} and Jensen's inequality,
	\begin{equation}\begin{split}
		\varrho^2((r(0), u(0)), (r(1), u(1)))
		&	=\int_0^1|r'|_\setdist^2(\tau)+\frac{|u'|_{W_2}^2(\tau)}{\diffusion}\ud{\tau} \\
		&	\ge\left(\int_0^1|r'|_\setdist(\tau)\ud{\tau}\right)^2
				{}+\frac{1}{\diffusion}\left(\int_0^1|u'|_{W_2}(\tau)\ud{\tau}\right)^2 \\
		&	\ge\varrho^2((r(0), u(0)), (r(1), u(1)))
	\end{split}\end{equation}
	which means that all inequalities are in fact equalities. Then $\tau\mapsto|r'|(\tau)$ and $\tau\mapsto|u'|(\tau)$ must be constant for almost all $\tau$. By the second (in)equality, these constants must be $\setdist(r(0), r(1))$ and $W_2(u(0), u(1))$, respectively. It follows that $\tau\mapsto r(\tau)$ and $\tau\mapsto u(\tau)$ are constant speed geodesics.
\end{proof}

With this lemma at hand, it would seem that the constant speed geodesics of $\setdist$ and $W_2$ can be studied separately. This is not the case, as $\States$ is a \emph{subspace} of $\Cells\times\Profiles$. After characterizing the constant speed geodesics of $\setdist$ and $W_2$, the resulting geodesic in $\Cells\times\Profiles$ will be studied.

Using the isometry \eqref{eq:ball_isometry}, it is easy to show that the constant speed geodesics of $\setdist$ are given by
\begin{equation}\label{eq:ball_geodesic}
	r(\tau)=\left((1-\tau)r(0)^\frac{n+1}{2}+\tau r(1)^\frac{n+1}{2}\right)^\frac{2}{n+1}\ge(1-\tau)r(0)+\tau r(1).
\end{equation}

It is shown in \cite[Theorem 7.2.2]{Ambrosio_Gigli_Savare} that the constant speed geodesics of the Wasserstein metric are given by
\begin{equation}\label{eq:Wasserstein_geodesic}
	u(\tau)=\left((1-\tau)\id+\tau t_{u(0)}^{u(1)}\right)_\#u(0),
\end{equation}
where $t_{u(0)}^{u(1)}$ is the optimal transport map from $u(0)$ to $u(1)$. This expression has a nice interpretation: the transport map $t_{u(0)}^{u(1)}$ tells where the mass at a certain position has to go to change the profile $u(0)$ into $u(1)$. Thus \eqref{eq:Wasserstein_geodesic} is the evolution where all mass travels from its initial position to its destination at a constant speed.

The characterization \eqref{eq:rearrangement} of the optimal transport map, which followes from the radial symmetry, leads to the following characterization of a constant speed geodesic, which is also shown in \cite[(7.2.8)]{Ambrosio_Gigli_Savare}
\begin{equation}\label{eq:geodesic_rearrangement}
	F_{u(\tau)}^{-1}=(1-\tau)F_{u(0)}^{-1}+\tau F_{u(1)}^{-1}
\end{equation}
where $F_u$ is defined as in \eqref{eq:cumulative_density}.

\bigskip

As noted above in Lemma \ref{th:states_geodesics}, the constant speed geodesics of $\Cells\times\Profiles$ curves $\tau\mapsto(r(\tau), u(\tau))$ with $r$ and $u$ as in \eqref{eq:ball_geodesic} and \eqref{eq:Wasserstein_geodesic}. By the inequality in \eqref{eq:ball_geodesic}, $(r(\tau), u(\tau))\in\States$ if this is the case for $\tau=0$ and $\tau=1$: geodesics between points in $\States$ do not leave $\States$. Note, however, that the geodesics have a peculiar property. Since the inequality in \eqref{eq:ball_geodesic} is strict for $0<\tau< 1$ unless $r(0)=r(1)$, the support of $u(\tau)$ will be strictly smaller than $B_{r(\tau)}$, even if $u(0)$ and $u(1)$ are positive throughout $B_{r(0)}$ and $B_{r(1)}$, respectively.

\subsection{$\lambda$-convexity}

As announced above, the concept of $\lambda$-convexity will be used. $\lambda$-convexity extends the notion of ordinary convexity, which will be equivalent to $0$-convexity. The number $\lambda$ can be interpreted as a measure of how convex a functional is. 

Lacking a linear structure, convexity of a functional on a metric space has to be defined using curves. A functional $\phi$ on a metric space $(\mathcal{S}, d)$ is said to be $\lambda$-convex along $\gamma:[0, 1]\to\mathcal{S}$ if
\begin{equation}\label{eq:lambda_convex_along_curve}
	\phi(\gamma(\tau))
		\le(1-\tau)\phi(\gamma(0))+\tau\phi(\gamma(1))-\frac{\lambda}{2}\tau(1-\tau)d^2(\gamma(0), \gamma(1))
\end{equation}
for all $\tau\in[0, 1]$. Clearly, it cannot be expected that a functional $\phi$ is $\lambda$-convex along \emph{all} curves. Inspired by convexity of a function on Euclidean space, which is equivalent to $0$-convexity along straight lines, one usually aks whether a functional is $\lambda$-convex along geodesics.

\bigskip

With the results from the previous section in mind, the convexity of $\Energy$ and $\varrho^2$ can be studied term by term. It will be shown that $\Energy$ is $\lambda$-convex along geodesics, and that $(r, u)\mapsto\frac{1}{2}\varrho^2((r, u), (s, w))$ is $1$-convex along geodesics.

First of all, $\Entropy$ is $0$-convex along geodesics if the map
\[
	z\mapsto z^n\integrand\left(z^{-n}\right)
\]
is convex and nonincreasing, as is noted in \cite[Proposition 9.3.9]{Ambrosio_Gigli_Savare} and was first shown by McCann \cite[Proposition 1.2]{McCann}. Similar to the proof of Lemma \ref{th:energy_coercivity}, this is implied by
\begin{equation}
	\diff{z}\left(z^n\integrand\left(z^{-n}\right)\right)=-n z^{n-1}\hat\integrand\left(z^{-n}\right),
\end{equation}
which is an increasing nonpositive function. It will turn out that this property of $\integrand$ also plays a role when computing the local slope of $\Entropy$.

The convexity of $W_2^2$ can easily be checked in the radially symmetric situation. From \eqref{eq:geodesic_rearrangement} and \eqref{eq:one_dim_wasserstein}, it follows immediately that
\begin{equation}\label{eq:wasserstein_convexity}
	W_2^2(u(\tau), w)
		=\int_{(0, 1)}|(1-\tau)U_0^{-1}(\sigma)+\tau U_1^{-1}(\sigma)-W^{-1}(\sigma)|^2\ud{\sigma}.
\end{equation}
A straightforward calculation shows that the map $x\mapsto\frac{1}{2}|x-y|^2$ is $1$-convex, as is also shown in \cite[Remark 2.4.4]{Ambrosio_Gigli_Savare}. Combining this with the above expression for $W_2^2$ implies that $w\mapsto\frac{1}{2}W_2^2(u, w)$ is $1$-convex along geodesics for any $u\in\Profiles$. Note that the special properties of $W_2$ for radially symmetric profiles are used in the proof. This is really necessary: in the general case, the Wasserstein distance is \emph{not} $1$-convex. A straightforward counterexample is given in \cite[Example 9.1.5]{Ambrosio_Gigli_Savare}. Actually, it is shown in \cite[Theorem 7.3.2]{Ambrosio_Gigli_Savare} that the opposite inequality holds.

The convexity of $\Perimeter$ can be checked using the isometry \eqref{eq:ball_isometry}: $\lambda$-convexity of $\Perimeter$ is equivalent to $\lambda$-convexity of the map
\[
	s\mapsto \theta_ns^\frac{2n-2}{n+1}
\]
where $\theta_n>0$ is a constant depending on the dimension. Note that the exponent is always between 0 and 2. The second derivative of this map is
\[
	s\mapsto\theta_n\frac{(2n-2)(n-3)}{(n+1)^2} s^\frac{-4}{n+1}.
\]
Using \cite[Remark 2.4.4]{Ambrosio_Gigli_Savare}, this means that $\Perimeter$ is $0$-convex if $n\ge 3$. Note that $\perimeter_2$ is not $\lambda$-convex along geodesics for any $\lambda$ since the second derivative is not bounded from below. However, if $s$ is bounded away from zero, the second derivative is bounded from below, and $\Perimeter$ is $\lambda$-convex, with $\lambda$ equal to the infimum of the second derivative. As was shown in the previous section, the radius $r$ is bounded away from $0$ in any sublevel of $\Energy$, which means that restricting the problem to a sublevel means that $\Perimeter$ is $\lambda$-convex for some $\lambda<0$. Note that a similar argument can be used for $n>3$ to obtain $\lambda>0$ on sublevels.

Finally, again using the isometry \eqref{eq:ball_isometry} and \cite[Remark 2.4.4]{Ambrosio_Gigli_Savare}, $s\mapsto\frac{1}{2}\setdist^2(r, s)$ is 1-convex along geodesics for all $r\in\Cells$.

Collecting the above results, one finds

\begin{lem}
	The map $(s, w)\mapsto\frac{1}{2}\varrho^2((r, u), (s, w))$ is $1$-convex along geodesics. Moreover, if $n\ge 3$, $\Energy$ is $0$-convex along geodesics. In case $n=2$, $\Energy$ is $\lambda$-convex along geodesics with $\lambda<0$ on its sublevels.
	
	In particular, Assumption \ref{as:convexity} holds. That is, given $(s, w), (r(0), u(0)), (r(1), u(1))$, the map
	\begin{equation}
		(r, u)\mapsto\Energy(r, u)+\frac{1}{2h}\varrho^2((s, w), (r, u))
	\end{equation}
	is convex along the constant speed geodesic $\tau\mapsto(r(\tau), u(\tau))$, defined by \eqref{eq:ball_geodesic}, \eqref{eq:Wasserstein_geodesic}.
\end{lem}

One can ask if the lack of convexity in case $n=2$ is due to the odd shape of the geodesics of $\setdist$. Obviously, the convexity of $\Perimeter$ is much better along linear interpolants than along geodesics: even if $n=2$, it is immediately clear that $\Perimeter$ is $0$-convex without having to restrict to a sublevel of $\Energy$. Unfortunately, $\setdist^2$ is not $1$-convex along linear interpolants. Moreover, the profit from replacing geodesics with linear interpolants is somewhat disappointing: even in case $n>3$, it is not possible to obtain $\lambda$-convexity for $\lambda>0$, which would be the metric equivalent of strict convexity. Therefore, it does not seem to be beneficial to use other curves instead of geodesics.

\bigskip

Having shown that Assumption \ref{as:convexity} holds, \cite[Theorems 2.4.15 and 4.0.4]{Ambrosio_Gigli_Savare} can be applied. Since the set of $(r, u)$ with $\Energy(r, u)<+\infty$ is dense in $\States$, no condition on the initial value is needed anymore.

\begin{thm}\label{th:convexity_results}
	For any $(r_0, u_0)\in\States$,
	\begin{itemize}
	\item	There exists a unique mimimizing movement $t\mapsto(r(t), u(t))$ for $\Energy$ starting at $(r_0, u_0)$.
	\item	$(r(t), u(t))$ is also the unique generalized minimizing movement for $\Energy$ starting at $(r_0, u_0)$.
	\item	$(r, u)$ is a locally Lipschitz curve of maximal slope with $|\partial\Energy|(r(t), u(t))<+\infty$ for $t>0$.
	\item $(r, u)$ is the unique solution of the evolution variational inequality
				\begin{equation}
					\frac{1}{2}\diff{t}\varrho^2((r(t), u(t)), (s, w))+\frac{\lambda}{2}\varrho^2((r(t), u(t)), (s, w))+\Energy(r(t), u(t))\le\Energy(s, w)
				\end{equation}
				for all $(s, w)\in\States$ such that $\Energy(s, w)<+\infty$ and almost all $t>0$.
	\item	For $(r_0, u_0)$ and $(s_0, w_0)$, the minimizing movements $(r, u)$ and $(s, w)$ for $\Energy$ starting at $(r_0, u_0)$ and $(s_0, w_0)$, respectively, satisfy
				\begin{equation}
					\varrho((r(t), u(t)), (s(t), w(t)))\le e^{-\lambda t}\varrho((r_0, u_0), (s_0, w_0)).
				\end{equation}
	\end{itemize}
	If $\Energy(r_0, u_0)<+\infty$, the right metric derivative
	\begin{equation}
		|(r, u)'_+|(t):=\lim_{h\downto 0}\frac{\varrho((r(t+h), u(t+h)), (r(t), u(t)))}{h},
	\end{equation}
	and the equation
	\begin{equation}\label{eq:strong_maximal_slope}
		\diff{t_+}\Energy(r(t), u(t))
			=-|\partial\Energy|^2(r(t), u(t))
			=-|(r, u)'_+|^2(t)
			=-|\partial\Energy|(r(t), u(t))|(r, u)'_+|(t)
	\end{equation}
	holds for all $t>0$. If, additionally, $|\partial\Energy|(r_0, u_0)<+\infty$, \eqref{eq:strong_maximal_slope} also holds for $t=0$.
\end{thm}
\begin{rem}
	The definitions of \emph{curve of maximal slope} and $|\partial\Energy|$ will be presented below. It will turn out that \eqref{eq:strong_maximal_slope} is stronger that the statement that $(r(t), u(t))$ is a curve of maximal slope
\end{rem}

Unfortunately, as $\lambda\le 0$, the contraction property does not give a lot of information about the asymptotic behaviour of the minimizing movement.

\section{Curves of maximal slope}
\label{s:MaxSlope}

One of the conclusions of Theorem \ref{th:convexity_results} is that for any $(r_0, u_0)\in\States$, a curve of maximal slope for $\Energy$ starting from $(r_0, u_0)$ exists. In this section, the concepts of \emph{local slope}, denoted by $|\partial\Energy|$ and \emph{curve of maximal slope} will be introduced. Additionally, the local slope of $\Energy$, which may be infinite, will be computed in this section.

A curve of maximal slope can be regarded as being the metric equivalent of a gradient flow. As before, the inspiration is the situation in $\Reals^n$. The equation $\dot{x}(t)=-\nabla\phi(x(t))$ flow for the gradient flow of some smooth $\phi:\Reals^n\to\Reals$ is equivalent to
\begin{equation}
	\diff{t}\phi(x(t))= -\frac{1}{2}|\nabla\phi(x(t))|-\frac{1}{2}|\dot{x}(t)|.
\end{equation}
Using Young's inequality, 
\begin{equation}
	\diff{t}\phi(y(t))
	\ge -|\nabla\phi(y(t))||\dot{y}(t)|
	\ge -\frac{1}{2}|\nabla\phi(y(t))|^2-\frac{1}{2}|\dot{y}(t)|^2
\end{equation}
for \emph{any} smooth curve $y$. Therefore, a gradient flow in $\Reals^n$ can be characterized by the inequality
\begin{equation}\label{eq:max_slope}
	\diff{t}\phi(x(t))\le -\frac{1}{2}|\nabla\phi(x(t))|^2-\frac{1}{2}|\dot{x}(t)|^2.
\end{equation}

If it is possible to redefine $|\nabla\phi|$ and $|\dot{x}|$ in a metric space such that the chain rule for $\phi\circ x$ still holds, this inequality can be used to define a gradient flow in a metric space. A generalization for $|\dot{x}|$ has already been found: it is the metric derivative defined in the previous section. A possible generalization for $|\nabla\phi|$ is the \emph{local slope}. The local slope of a functional $\phi:\mathcal{S}\to\Reals$ is defined by
\begin{equation}\label{eq:def_local_slope}
	|\partial\phi|(v)=\limsup_{w\to v}\frac{(\phi(v)-\phi(w))^+}{d(v, w)}.
\end{equation}
Note that the local slope is in some sense one-sided: it only measures how fast the value of $\phi$ decreases near a point. This makes sense, since the functionals that are considered are usually only lower semicontinuous. See \cite[\S1]{Ambrosio_Gigli_Savare} for more details.

The local slope of $\Energy$ can be computed using methods from the proof of \cite[Theorem 10.4.6]{Ambrosio_Gigli_Savare}. The main ingredient is \cite[Lemma 10.4.4]{Ambrosio_Gigli_Savare}, which describes the behavior of $\Entropy(u)$ as $u$ is pushed forward by a sufficiently smooth map $\mathbf{r}$.

\begin{lem}\label{th:push_forward_entropy}
	Let $u\in\Profiles$, $\mathbf{r}\in L^2(\Reals^n, u)$ be a radial vector field, $\overline{\tau}>0$ and suppose that
	\begin{enumerate}
	\item	$\mathbf{r}$ is differentiable $u$-almost everywhere, and $\mathbf{r}_t:=(1-\tau)\mathbf{r}+\tau\id$ is $u\Lebesgue^n$-injective with $|\det\nabla r_\tau|>0$ $u\Lebesgue^n$-almost everywhere for any $\tau\in[0, \overline{\tau}]$,
	\item $\|\nabla(\mathbf{r}-\id)\|_{L^\infty(\Reals^n, u)}<+\infty$
	\item There exists a constant $C>0$ such that $\forall z_1, z_2:\integrand(z_1+z_2)\le C\left(1+\integrand(z_1)+\integrand(z_2)\right)$ \label{as:doubling}
	\item	$\Entropy((\mathbf{r}_\tau)_\#u)<+\infty$.
	\end{enumerate}
	Then the map $\tau\mapsto\tau^{-1}(\Entropy((\mathbf{r}_\tau)_\#u)-\Entropy(u))$ is nondecreasing in $[0, \overline{\tau}]$, and
	\begin{equation}\label{eq:push_forward_entropy}
		+\infty>\lim_{\tau\downto 0}\frac{\Entropy((\mathbf{r}_\tau)_\#u)-\Entropy(u)}{\tau}=-\int_{\Reals^n}\hat\integrand(u(x))\tr\nabla(\mathbf{r}(x)-x)\ud{x}
	\end{equation}
\end{lem}
It is easy to check that $z\mapsto z\log z$ satisfies this satisfies Condition \eqref{as:doubling}, also known as the 'doubling' condition. Moreover, it does not limit the growth of $\integrand$ too much. For instance, the functional $z\mapsto z^2$ also satisfies Condition \eqref{as:doubling}. The conditions on $\mathbf{r}$ are a bit technical, but there are two important examples of vector fields that satisfy the requirements: smooth, injective vector fields and optimal transport maps between profiles with copmact support. The latter follows from \cite[Theorem 6.2.7]{Ambrosio_Gigli_Savare} about regularity of optimal transport maps, and $u\Lebesgue^n$-essential injectivity follows as in the beginning of the proof of \cite[Proposition 9.3.9]{Ambrosio_Gigli_Savare}. Note that the apprximate differential in \cite[Theorem 6.2.7]{Ambrosio_Gigli_Savare} is not needed if $\mathbf{r}$ is a transport map between compactly supported probability measures.

Before the main theorem can be proven, a refined version of \cite[Lemma 10.4.5]{Ambrosio_Gigli_Savare} is needed. This lemma shows that, if $\hat\integrand$ is smooth enough, a weak integration by parts formula holds.

\begin{lem}\label{th:slope_ibp}
	Let $(r, u)\in\States$, and $\boldsymbol{r}$ and $\overline{\tau}$ be as in Lemma \ref{th:push_forward_entropy}. In addition, suppose that
	\begin{enumerate}
	\item	$\hat\integrand(u)|_{B_r}\in W^{1, 1}(B_r),$
	\item $\mathbf{r}$ is of bounded variation.
	\end{enumerate}
	Then
	\begin{equation}
		\int_{B_r} \hat\integrand(u(x))\Div(\boldsymbol{r}(x)-x)\ud{x}
			\le -\int_{B_r} \nabla\hat\integrand(u(x))\cdot(\boldsymbol{r}(x)-x)\ud{x}
				{}+(s-r)\hat\integrand(u(r))\Perimeter(r)
	\end{equation}
\end{lem}
\begin{proof}
	Without loss of generality, it can be assumed that $\overline{\tau}=1$. As in the proof of \cite[Lemma 10.4.5]{Ambrosio_Gigli_Savare}, the main ingredient is an estimate for the distributional divergence of a BV vector field $\boldsymbol{\eta}$. However, now using \cite[3.92]{Ambrosio_Fusco_Pallara}, the boundary term at $\bdry B_r$ should also be taken into account. If $w\in C^\infty_c(\Reals^n)$ is nonnegative and radially symmetric, its trace on $\bdry B_r$ is constant, say $w(x)=\tilde{w}$ if $|x|=r$. Then
	\begin{equation}\label{eq:ibp_inequality_basic}
		\int_{B_r} w(x)\Div(\boldsymbol{\eta}(x))\ud{x}
			\le -\int_{B_r}\nabla w(x)\cdot\boldsymbol{\eta}(x)\ud{x}+\boldsymbol{\eta}(r)\tilde{w}\Perimeter(r)
	\end{equation}
	if the distributional divergence of $\boldsymbol{\eta}$ is a nonnegative measure. For bounded $\boldsymbol{\eta}$, the same inequality holds for nonnegative and radially symmetric $w$ such that $w|_{B_r}\in W^{1, 1}(B_r)$ and $w\equiv 0$ outside $B_r$. Applying this for $w=\hat\integrand(u)$, $\boldsymbol{\eta}:=\boldsymbol{r}$,
	\begin{equation}\begin{split}\label{eq:vector_field_ibp}
		\int_{B_r} \hat\integrand(u(x))\Div(\boldsymbol{r}(x))\ud{x}
		&	\le -\int_{B_r}\nabla\hat\integrand(u(x))\cdot\boldsymbol{r}(x)\ud{x}
						{}+\boldsymbol{r}(r)\hat\integrand(u(r))\Perimeter(r) \\
		&	\le -\int_{B_r}\nabla\hat\integrand(u(x))\cdot\boldsymbol{r}(x)\ud{x}+s\hat\integrand(u(r))\Perimeter(r).
	\end{split}\end{equation}
	Standard integration by parts yields
	\begin{equation}
		\int_{B_r} \hat\integrand(u(x))\Div(x)\ud{x}
			= -\int_{B_r}\nabla\hat\integrand(u(x))\cdot x\ud{x}+r\hat\integrand(u(r))\Perimeter(r),
	\end{equation}
	which, together with \eqref{eq:vector_field_ibp}, yields the result.
\end{proof}

Comparing this result to \cite[Lemma 10.4.5]{Ambrosio_Gigli_Savare}, the approximating sequence $\mathbf{r}_k$ is left out. This can be done because the lemma will be applied to an optimal transport map between radially symmetric measures. From \eqref{eq:rearrangement}, it can be seen that $t_u^w$ is always monotone. Hence, the map $\mathbf{t}_u^w(x):=t_u^w(x)\frac{x}{|x|}$, for which the lemma will be applied, is of bounded variation.

Apart from radial symmetry, the only difference between this proof and the proof of \cite[Lemma 10.4.5]{Ambrosio_Gigli_Savare} is from the boundary term in \eqref{eq:ibp_inequality_basic}. In the following computations, more boundary terms will appear in a similar way. Loosely speaking, $W^{1, 1}(\Reals^n)$ functions in \cite[\S10.4]{Ambrosio_Gigli_Savare} are replaced by functions that are $W^{1, 1}$ when restricted to $B_r$. In the language of $\BV$ functions, this means that the distributional derivative consists of an absolutely continuous part and a jump on $\bdry B_r$. Obviously, it is not important that the domain is a ball when performing these calculations: the same calculations can be done on much more general domains.

Since the calculation is basically integration by parts, it is not surprising that boundary terms appear when restricting to a bounded domain. Through the calculations that follow, however, the boundary term will give the osmotic term that was discussed in the introduction. This will become clear in the proof of the following theorem that gives the local slope $|\partial\Energy|$ of $\Energy$.

\begin{thm}\label{th:local_slope}
	Let $(r, u)\in\States$ be given. Then $\Energy$ has finite local slope at $(r, u)$ if and only if $\hat\integrand(u)\in W^{1, 1}((0, r), \Perimeter)$ with 
	\begin{equation}
		\frac{\nabla\hat\integrand(u)}{u}\in L^2(\Reals^n, u\Perimeter).
	\end{equation}
	In this case,
	\begin{equation}\label{eq:local_slope_expression}
		|\partial\Energy|(r, u)
			=\left(\left|\frac{n-1}{r}-\hat\integrand(u(r))\right|^2\Perimeter(r)
				{}+\diffusion\left\|\frac{\nabla\hat\integrand(u)}{u}\right\|^2_{L^2(\Reals^n, u)}\right)^\frac{1}{2}.
	\end{equation}
\end{thm}
\begin{proof}
	Suppose first that $|\partial\Energy|(r, u)$ is finite. Let $\boldsymbol{\xi}$ be a radial vector field, and define $\xi$ by $\boldsymbol{\xi}(x)=\xi(|x|)\frac{x}{|x|}$. Similar to the proof of \cite[10.4.6]{Ambrosio_Gigli_Savare}, note that
	\begin{equation}
		W_2\left(u, (X_\tau)_\#u\right)
			\le\tau\|\boldsymbol{\xi}\|_{L^2(\Reals^n, u)}
			=\tau\|\xi\|_{L^2((0, \infty), u\Perimeter)}
	\end{equation}
	where $X_\tau(x)=x+\tau\boldsymbol{\xi}(x)$. Using the definitions of $|\partial\Energy|$ and metric derivative, followed by this inequality and \eqref{eq:set_dist_derivative}, 
	\begin{equation}\label{eq:EL_inequality}\begin{split}
		\left|\int_{\Reals^n}\hat\integrand(u(x))\Div\boldsymbol{\xi}(x)\ud{x}\right|
		&	\le\lim_{\tau\downto 0}\frac{|\Energy(X_\tau(r), (X_\tau)_\#u)-\Energy(r, u)|}{\tau}
					{}+\frac{|\Perimeter(X_\tau(r))-\Perimeter(r)|}{\tau} \\
		&	\le|\partial\Energy|(r, u)\left|\left(X_\tau(r), (X_\tau)_\#u\right)\right|'(0)
					+|\xi(r)|\Perimeter'(r) \\
		&	\le|\partial\Energy|(r, u)
					\left(|\xi(r)|^2\Perimeter(r)
						{}+\frac{\|\xi\|^2_{L^2((0, \infty), u\Perimeter)}}{\diffusion}\right)^\frac{1}{2}
					{}+|\xi(r)|\Perimeter'(r)
	\end{split}\end{equation}
	whenever $\xi$ is smooth enough to apply Lemma \ref{th:push_forward_entropy} with $\boldsymbol{r}=\boldsymbol{\xi}+\id$. For $\boldsymbol{\xi}=-\xi$, this yields
	\begin{equation}
		n\int_{\Reals^n}\hat\integrand(u(x))\ud{x}
			\le|\partial\Energy|(r, u)\left(r^2\Perimeter(r)
					{}+\frac{\|\id\|^2_{L^2((0, \infty), u\Perimeter)}}{\diffusion}\right)^\frac{1}{2}+r\Perimeter'(r),
	\end{equation}
	which means that $\hat\integrand(u)$ is integrable. Next, apply \eqref{eq:EL_inequality} for smooth, compactly supported radial $\boldsymbol{\xi}$ to obtain
	\begin{equation}\begin{split}
		\left|\int_{\Reals^n}\hat\integrand(u(x))\Div\boldsymbol{\xi}(x)\ud{x}\right|
		&	\le|\partial\Energy|(r, u)
					\left(|\xi(r)|^2\Perimeter(r)
						{}+\frac{\|\xi\|^2_{L^2((0, \infty), u\Perimeter)}}{\diffusion}\right)^\frac{1}{2}
					{}+|\xi(r)|\Perimeter'(r) \\
		&	\le\left(|\partial\Energy|(r, u)\left(\frac{1}{\diffusion}+\Perimeter\right)^\frac{1}{2}+\Perimeter'(r)\right)
					\|\boldsymbol{\xi}\|_{C^0((0, \infty))}.
	\end{split}\end{equation}
	Noting that this inequality only needs to be checked for radial $\boldsymbol{\xi}$, Riesz' theorem \cite[1.54]{Ambrosio_Fusco_Pallara} implies that $x\mapsto\hat\integrand(u(x))$ is of bounded variation. Hence,
	\begin{equation}
		\int_{\Reals^n}\hat\integrand(u(x))\Div\boldsymbol{\xi}(x)\ud{x}
			=-\int_{\Reals^n}\boldsymbol{\xi}(x)\ud{D\hat\integrand(u(x))}.
	\end{equation}
	Restricting to $\xi$ with $\xi(r)=0$,
	\begin{equation}\label{eq:perturbation_boundary_fixed}
		\left|\int_{B_r}\boldsymbol{\xi}\ud{D\hat\integrand(u)}\right|
			=\left|\int_{B_r}\hat\integrand(u(x))\Div\boldsymbol{\xi}(x)\ud{x}\right|
			\le\frac{|\partial\Energy|(r, u)}{\diffusion}\|\boldsymbol{\xi}\|^2_{L^2(\Reals^n, u)}
	\end{equation}
	Using duality, this means that $D\hat\integrand(u)$ is absolutely continuous on $B_r$.	In particular, $\hat\integrand(u)|_{B_r}\in W^{1, 1}(B_r)$. Since $u\equiv 0$ outside $B_r$, it follows that the singular part of $D\hat\integrand(u)$ must be concentrated on $\bdry B_r$. By radial symmetry, it follows from \cite[Proposition 3.92]{Ambrosio_Fusco_Pallara} that $D\hat\integrand(u)$ only consist of an absolutely continuous part concentrated on $B_r$ and possibly a jump part accross $\bdry B_r$. More precisely, 
	\begin{equation}\begin{split}\label{eq:BV_derivative_flux}
		\int_{\Reals^n}\hat\integrand(u(x))\Div\boldsymbol{\xi}(x)r\ud{x}
		&	=-\int_{B_r}\nabla\hat\integrand(u(x))\cdot\boldsymbol{\xi}(x)\ud{x}
				{}+\int_{\bdry B_r}\hat\integrand(u(x))\boldsymbol{\xi}(x)\cdot\frac{x}{r}\ud{\Hausdorff^{n-1}(x)} \\
		&	=-\int_0^\infty\diff{\rho}\hat\integrand(u(\rho))\xi(\rho)\Perimeter(\rho)\ud{\rho}
				{}+\hat\integrand(u(r))\xi(r)\Perimeter(r).
	\end{split}\end{equation}
	
	Using Lemma \ref{th:push_forward_entropy}, and \eqref{eq:BV_derivative_flux}
	\begin{equation}\begin{split}
		\lim_{\tau\downto 0}\frac{\Energy(X_\tau(r), (X_\tau)_\#u)-\Energy(r, u)}{\tau}
		&	=\xi(r)\Perimeter'(r)-\int_{B_r}\hat\integrand(u(x))\Div\boldsymbol{\xi}(x)\ud{x} \\
		&	=\left(\Perimeter'(r)-\hat\integrand(u(r))\Perimeter(r)\right)\xi(r)
			+\int_{B_r}\nabla\hat\integrand(u(x))\cdot\boldsymbol{\xi}(x)\ud{x}.
	\end{split}\end{equation}
	Estimating the left hand side as in \eqref{eq:EL_inequality},
	\begin{multline*}
		\left|\left(\frac{\Perimeter'(r)}{\Perimeter(r)}-\hat\integrand(u(r))\right)\Perimeter(r)\xi(r)
			{}+\int_{B_r}\frac{\nabla\hat\integrand(u(x))}{u(x)}\cdot\boldsymbol{\xi}(x)u(x)\ud{x}\right| \\
		\le|\partial\Energy|(r, u)\left(|\xi(r)|^2\Perimeter(r)
			{}+\frac{\|\boldsymbol{\xi}\|^2_{L^2(\Reals^n, u)}}{\diffusion}\right)^\frac{1}{2}.
	\end{multline*}
	Finally, using another duality argument and evaluating $\frac{\Perimeter'(r)}{\Perimeter(r)}$,
	\begin{equation}
		\left(\left|\frac{n-1}{r}
									{}-\hat\integrand(u(r))\right|^2\Perimeter(r)
			{}+\diffusion\int_{B_r}\left|\frac{\nabla\hat\integrand(u(x))}{u(x)}\right|^2u(x)\ud{x}\right)^\frac{1}{2}
			\le|\partial\Energy|(r, u).
	\end{equation}
	\bigskip
	
	Conversely, assume $\hat\integrand(u)|_{B_r}\in W^{1, 1}(B_r)$, 
	\begin{equation}
		\frac{\nabla\hat\integrand(u)}{u}\in L^2(\Reals^n, u),
	\end{equation}
	and let $(s, w)\in\States$ such that $\Energy(s, w)\le\Energy(r, u)$. As noted above, the map $\boldsymbol{t}_u^w(x)=t_u^w(x)\frac{x}{|x|}$, which is the optimal transport map from $u$ to $w$ in $\Reals^n$, is sufficiently regular to apply Lemmata \ref{th:push_forward_entropy} and \ref{th:slope_ibp}. Then
	\begin{equation}\begin{split}
		\Energy(s, w)-\Energy(r, u)
		&	\ge(\Perimeter(s)-\Perimeter(r))
				{}-\int_{\Reals^n}\hat\integrand(u(x))\tr\nabla(\boldsymbol{t}_u^w(x)-x)\ud{x} \\
		&	\ge\left(\frac{\Perimeter'(r)}{\Perimeter(r)}-\hat\integrand(u(r))\right)\left(s-r\right)\Perimeter(r)
				{}+\int_{\Reals^n}\nabla\hat\integrand(u(x))\cdot(\boldsymbol{t}_u^w(x)-x)\ud{x} \\
		&	\ge -\left(\left|\frac{n-1}{r}-\hat\integrand(u(r))\right|^2\Perimeter(r)
				{}+\diffusion\left\|\frac{\nabla\hat\integrand(u)}{u}\right\|^2_{L^2(\Reals^n, u)}
				\right)^\frac{1}{2} \\
		&\hspace{1cm}		\left(|s-r|^2\Perimeter(r)+\frac{W_2^2(u, w)}{\diffusion}\right)^\frac{1}{2}
	\end{split}\end{equation}
	using convexity of $r\mapsto\Perimeter(r)$. The inequality in \eqref{eq:local_slope_expression} left to prove now follows by writing a Taylor expansion for the integral from \eqref{eq:def_set_dist}, and taking the limit superior for $(s, w)\to(r, u)$.
\end{proof}

Note again that boundary terms appear, most importantly in \eqref{eq:BV_derivative_flux}. As explained above, this is a direct consequence of restricting $u$ to a bounded domain. The fact that the boundary of the domain moves is not important here yet. In \eqref{eq:perturbation_boundary_fixed}, it becomes clear that the boundary term disappears when the perturbation is such that the mass at the boundary does not move. Studying the calculations a bit closer, the boundary term can be `estimated away' if $\xi(r)<0$. Hence, the boundary term still plays a role when the boundary is fixed, but disappears from the final outcome. From a modelling point of view this makes sense: if the cell membrane would be fixed for some reason, the osmotic force would still be there. The only thing that has changed is that te membrane does not react to forces anymore. In a more general setting, a fixed boundary can be regarded as a moving boundary with infinite resistance to force: it can be seen from the calculation that all forces are there, but disappear because the boundary is unable to react to them.

\bigskip

The section is concluded with a lemma that essentially characterizes the minimal Fr\'echet subdifferential of $\Energy$. 

\begin{lem}\label{th:frechet_derivative}
	Suppose that $(r(\tau), u(\tau))$ is an absolutely continuous curve in $\States$, and $v(\tau)$ is the solution of \eqref{eq:continuity_equation} with $\|v_\tau\|_{L^2((0, \infty), u(\tau)\Perimeter)}=|u'|(\tau)$ for almost every $\tau$. For any $\tau$ such that $|\partial\Energy|(r(\tau), u(\tau))<+\infty$, $\tau\mapsto\Energy(r(\tau), u(\tau))$ is differentiable, and \eqref{eq:optimal_map_velocity} holds,
	\begin{equation}
		\diff{\tau}\Energy(r(\tau), u(\tau))
			=\left(\frac{n-1}{r(\tau)}-\hat\integrand(u(r(\tau), \tau))\right)\Perimeter(r(\tau))r'(\tau)
				{}+\int_0^{r(\tau)}\pdiff[\hat\integrand(u(\rho, \tau))]{\rho}v_\tau(\rho)\Perimeter(\rho)\ud{\rho}.
	\end{equation}
\end{lem}
\begin{proof}
	For $h>0$, let $t_h$ be the optimal transport map from $u_t$ to $u_{t+h}$. Then, as in the proof of the Theorem \ref{th:local_slope},
	\begin{multline}
		\Energy(r(\tau+h), u(\tau+h))-\Energy(r(\tau), u(\tau)) \\
			\ge\left(\frac{n-1}{r(\tau)}-\hat\integrand(u(r(\tau), \tau))\right)
							\left(r(\tau+h)-r(\tau)\right)\Perimeter(r(\tau)) \\
					{}+\int_0^{r(\tau)}\pdiff[\hat\integrand(u(\rho, \tau))]{\rho}
							{}\cdot(t_{u(\tau)}^{u(\tau+h)}(\rho)-\rho)\Perimeter(\rho)\ud{\rho}
	\end{multline}
	Using \eqref{eq:optimal_map_velocity}, dividing by $h$, and taking limits for $h\downto 0$ yields
	\begin{equation}
		\diff{\tau}\Energy(r(\tau), u(\tau))
			\ge\left(\frac{n-1}{r(\tau)}-\hat\integrand(u(r(\tau), \tau))\right)\Perimeter(r(\tau))r'(\tau)
				{}+\int_0^{r(\tau)}\pdiff[\hat\integrand(u(\rho, \tau))]{\rho}v_\tau(\rho)\Perimeter(\rho)\ud{\rho}
	\end{equation}
	where $v$ is characterized by Remark \ref{rem:minimal_norm_solution}. The converse inequality follows by studying the left derivative.
\end{proof}

\section{Weak Solutions}
\label{s:WeakSol}

A natural question is whether a maximal slope curve obtained in Theorem \ref{th:convexity_results} is a solution of \eqref{eq:osmosis_scaled}, and in what sense. Bearing in mind Lemma \ref{th:frechet_derivative}, it is expected that the argument from the beginning of the previous section justifying the definition of a curve of maximal slope can in this case be reversed to obtain a weak form of \eqref{eq:osmosis_scaled}. Of course, the choice for $\integrand$ should be kept in mind:
\begin{equation}
	\left\{\begin{aligned}
		u_t &																																	=\diffusion\Laplace\hat\integrand(u), &\qquad&	\text{for $x\in B_{r(t)}$, $t>0$}, \\
		-\diffusion\left.\diff{\rho}\hat\integrand(u(\rho))\right|_{\rho=r} &	=u(r)r'(t), &&																	\text{for $t>0$} \\
		r' &																																	=-\frac{n-1}{r}+\hat\integrand(u(r)) &&					\text{for $t>0$.}
	\end{aligned}\right.
\end{equation}

It is possible to formulate this problem in a neater way. Note that
\begin{equation}
	\nabla\hat\integrand(u)
		=\nabla\left(u\integrand'(u)\right)-\nabla\integrand(u)
		=\integrand'(u)\nabla u+u\nabla\integrand'(u)-\integrand'(u)\nabla u
		=u\nabla\integrand'(u)
\end{equation}
if $\integrand$ is sufficiently smooth. Therefore, the first equation in \eqref{eq:osmosis_generalized} can be replaced by $u_t=\diffusion\Laplace\hat\integrand(u)$. Similarly, the boundary condition for $u$ can be rewritten:
\begin{equation}\label{eq:osmosis_generalized}
	\left\{\begin{aligned}
		u_t &																																=\diffusion\Div\left(u\nabla\integrand'(u)\right), &\qquad&	\text{for $x\in B_{r(t)}$, $t>0$}, \\
		-\diffusion\left.\diff{\rho}\integrand'(u(\rho))\right|_{\rho=r} &	=r'(t), &&																									\text{for $t>0$} \\
		r' &																																=-\frac{n-1}{r}+\hat\integrand(u(r)) &&											\text{for $t>0$.}
	\end{aligned}\right.
\end{equation}
This way of writing the problem makes the structure clearer: particles have velocity equal to $\nabla\integrand'(u)$, and the velocity of the particles at the boundary should match the normal velocity of the boundary.

In order to show that a curve of maximal slope can be regarded a weak solution of \eqref{eq:osmosis_generalized}, some integral identities are derived that characterize solutions when smoothness is assumed. By smooth, it is meant that $t\mapsto r(t)$ is continuously differentiable, and $(x, t)\mapsto u(x, t)$ is twice continuously differentiable on the domain
\[
	\Set{(\rho, t):0<\rho<r(t); t>0}
\]
Moreover, it is shown that $\Energy$ is a Lyapunov functional for \eqref{eq:osmosis_generalized}. As before, when applying the divergence rule, $u(t)$ is interpreted as a function on $\Reals^n$.

If $u$ is smooth in the sense explained above, and $\varphi\in C^\infty_c(\Reals^n\times(0, \infty))$ is radially symmetric,
\begin{equation}\begin{split}
	-\int_0^\infty\int_{B_{r(t)}}u_t(x, t)\varphi(x, t)\ud{x}\ud{t}
	&	=\int_0^\infty\int_{B_{r_t}}u_t(x, t)\varphi_t(x, t)\ud{x}\ud{t} \\
	&	\hspace{1cm}	{}+\int_0^\infty u(r(t), t)\tilde\varphi(t)r'(t)\Perimeter(r(t))\ud{t},
\end{split}\end{equation}
where $\tilde\varphi(t)$ is the value of $\varphi(., t)$ on $\bdry B_{r(t)}$. Moreover,
\begin{equation}\begin{split}
	-\int_0^\infty\int_{B_{r(t)}}u_t(x, t)\varphi(x, t)\ud{x}\ud{t}
	&	=-\diffusion\int_0^\infty\int_{B_{r(t)}}\Laplace\hat\integrand(u(x, t))\varphi(x, t)\ud{x}\ud{t} \\
	&	=\diffusion\int_0^\infty\int_{B_{r(t)}}u(x, t)\nabla\hat\integrand(u(x, t))\cdot\nabla\varphi(x, t)\ud{x}\ud{t} \\
	&	\hspace{1cm}	{}-\int_0^\infty\diffusion\left.\diff{\rho}\hat\integrand(u(\rho, t))\right|_{\rho=r}\tilde\varphi(t)\Perimeter(r(t))\ud{t}
\end{split}\end{equation}
Combining these two identities with the Neumann boundary condition yields
\begin{equation}\label{eq:weak_diffusion}\begin{split}
	\frac{1}{\diffusion}\int_0^\infty\int_{B_{r(t)}}u(x, t)\varphi_t(x, t)\ud{x}\ud{t}
	&	=\int_0^\infty\int_{B_{r(t)}}\nabla\hat\integrand(u(x, t))\cdot\nabla\varphi(x, t)\ud{x}\ud{t}.
\end{split}\end{equation}
By a similar computation, for any radially symmetric $\psi\in C^\infty_c(\Reals^n\times(0, \infty))$,
\begin{equation}\label{eq:weak_boundary}\begin{split}
	\int_0^\infty\int_{B_{r(t)}}\psi_t(x, t)\ud{x}\ud{t}
	&	=-\int_0^\infty\tilde\psi(t)r'(t)\Perimeter(r(t))\ud{t} \\
	&	=\int_0^\infty\left(\frac{n-1}{r(t)}-\hat\integrand(u(r(t), t))\right)\tilde\psi(t)\Perimeter(r(t))\ud{t},
\end{split}\end{equation}
where $\tilde\psi$ is defined analogous to $\tilde\varphi$. From these identities, it follows that $\Energy$ is indeed a Lyapunov functional for \eqref{eq:osmosis_generalized}:
\begin{gather*}
	\diff{t}\left(\Perimeter(r(t))+\int_{B_{r(t)}}\integrand(u(x, t))\ud{x}\right) \hspace{7cm}\\\begin{split}\hspace{2cm}
	&	=\Perimeter'(r(t))r'(t)+\int_{B_{r_t}}\integrand'(u(x, t))u_t(x, t)\ud{x}
				{}+\integrand(u(r(t), t))r'(t)\Perimeter(r(t)) \\
	& =\left(\frac{\Perimeter'(r(t))}{\Perimeter(r(t))}-\hat\integrand(u(r(t), t))\right)r'(t)\Perimeter(r(t)) \\
	&\hspace{1cm}	{}+\diffusion\int_{B_{r(t)}}\integrand'(u(x, t))\Div\left(u(x, t)\nabla\integrand'(u(x, t))\right)\ud{x} \\
	&\hspace{1cm}	{}+u(r(t), t)\integrand'(u(r(t), t))r'(t)\Perimeter(r(t)) \\
	& =-\left(\frac{n-1}{r(t)}-\hat\integrand(u(r(t), t))\right)^2\Perimeter(r(t))
				{}+\diffusion\int_{B_{r(t)}}|\nabla\integrand'(u(x, t)|^2u(x, t)\ud{x} \\
	&\hspace{1cm}	{}+\integrand'(u(r(t), t))\Perimeter(r(t))
									\left(u(r(t), t)r'(t)+\diffusion\left.\diff{\rho}\hat\integrand(u(\rho, t))\right|_{\rho=r(t)}\right) \\
	& =-\left(\frac{n-1}{r(t)}-\hat\integrand(u(r(t), t))\right)^2\Perimeter(r(t))
				{}-\diffusion\int_{B_r(t)}\frac{|\nabla\hat\integrand(u(x, t))|^2}{u(x, t)}\ud{x},
\end{split}\end{gather*}
since 
\begin{equation}
	u\left[\integrand'\circ u\right]'=\left[\hat\integrand\right]'.
\end{equation}

If the integral identities \eqref{eq:weak_diffusion}, \eqref{eq:weak_boundary}, together with monotonicity of $\Energy$ along a trajectory are accepted as a definition of weak solution of \eqref{eq:osmosis_generalized}, the following theorem simply states that curves of maximal slope of $\Energy$ are weak solutions of \eqref{eq:osmosis_generalized}. Loosely speaking, this means that \eqref{eq:osmosis_generalized} can be formulated as a gradient flow.

\begin{thm}\label{th:weak_solution}
	Suppose that $t\mapsto (r(t), u(t))\in\States$ is continuous. Then $(r(t), u(t))$ is a curve of maximal slope for $\Energy$ if and only if $\Energy(r(t), u(t))$ is decreasing, and \eqref{eq:weak_diffusion} and \eqref{eq:weak_boundary} hold for all radially symmetric test functions $\phi, \psi\in C^\infty_c(\Reals^n\times(0, \infty))$.
\end{thm}
\begin{proof}
	Assume first that $(r(t), u(t))$ is a curve of maximal slope for $\Energy$. By definition, this means that $\phi(t):=\Energy(r(t), u(t))$ is decreasing, and $t\mapsto u(t)$ and $t\mapsto r(t)$ are both absolutely continuous. Then there exists a function $v$ satisfying \eqref{eq:continuity_equation} and $\|v_\tau\|_{L^2((0, \infty), u\Perimeter)}=|u'|(t)$ for almost every $t>0$. Hence, as in the proof of Theorem \ref{th:continuity_equation}, $\mathbf{v}_\tau(x):=v_\tau(|x|)\frac{x}{|x|}$ solves
	\begin{equation}\label{eq:u_derivative}
		\int_0^\infty\int_{B_{r(t)}}\varphi_t(x, t)u(x, t)\ud{x}\ud{t}
			=-\int_0^\infty\int_{B_{r(t)}}\nabla\varphi_t\cdot\mathbf{v}(x, t)u(x, t)\ud{x}\ud{t}
	\end{equation}
	for any radially symmetric $\varphi\in C^\infty_c(\Reals^n\times(0, \infty))$. Using Lemma \ref{th:frechet_derivative},
	\begin{equation}
		\phi'(t)=\left(\frac{n-1}{r_t}-\hat\integrand(u_t(r_t))\right)\Perimeter(r_t)\diff[r_t]{t}
				{}+\int_0^{r(\tau)}\pdiff[\hat\integrand(u(\rho, \tau))]{\rho}v_\tau(\rho)\Perimeter(\rho)\ud{\rho}
	\end{equation}
	for almost all $t>0$, where $\phi(t):=\Energy(r(t), u(t))$. On the other hand,
	\begin{equation}\begin{split}
		\phi'(t)
		&	\le-\frac{|(r, u)'|^2(t)}{2}-\frac{|\partial\Energy|^2(r(t), u(t))}{2} \\
		&	=-\frac{1}{2}\left(\left|r'(t)\right|^2\Perimeter(r(t))
				{}+\left|\frac{n-1}{r(t)}-\hat\integrand(u(r(t), t))\right|^2\Perimeter(r(t))\right. \\
		&	\hspace{2.5cm}	\left.{}+\frac{\|\mathbf{v}_t\|^2_{L^2(\Reals^n, u(t))}}{\diffusion}
				{}+\diffusion\left\|\frac{\nabla\hat\integrand(u(x, t))}{u(x, t)}\right\|^2_{L^2(\Reals^n, u(t))}\right)
	\end{split}\end{equation}
	for almost all $t>0$. By Young's equality, it follows that for all $t>0$ such that both hold,
	\begin{gather*}
		r'(t)=-\frac{n-1}{r(t)}+\hat\integrand(u(r(t), t)), \\
		\mathbf{v}_t=-\diffusion\frac{\nabla\hat\integrand(u(t))}{u(t)}.
	\end{gather*}
	Substituting the latter in \eqref{eq:u_derivative} yields \eqref{eq:weak_diffusion}. Together with \eqref{eq:ball_continuity}, the first implies \eqref{eq:weak_boundary}.
	
	Conversely, note that \eqref{eq:weak_diffusion} is equivalent to \eqref{eq:continuity_equation} with a given expression for $v(t)$ substituted. Therefore, $t\mapsto u(t)$ is absolutely continuous with respect to the Wasserstein metric. Similarly, by Lemma \ref{th:ball_continuity}, $t\mapsto r(t)$ is absolutely continous, and gives and expression for $r'(t)$ almost everywhere. Then Lemma \ref{th:frechet_derivative} yields
	\begin{equation}
		\phi'(t)
			=\left(\frac{n-1}{r(t)}-\hat\integrand(u(r(t), t))\right)\Perimeter(r(t))r'(t)
				{}+\int_0^{r(\tau)}\pdiff[\hat\integrand(u(\rho, \tau))]{\rho}v_\tau(\rho)\Perimeter(\rho)\ud{\rho}
	\end{equation}
	for almost every $t>0$. 
	
	Having explicit expressions for $\diff[r_t]{t}$ and $v_t$ at hand, it follows from Young's inequality that $t\mapsto(r(t), u(t))$ is a curve of maximal slope.
\end{proof}

Note that the splitting of the problem from the introduction plays an important role: the equations related to absolute continuity give the equations relating the evolution of $u$ and $r$ to different velocities, whereas the gradients give the velocity in terms of the current state. As argued above, the divergence formulation \eqref{eq:osmosis_generalized}, also has this structure.

\section{Varying permeability}
\label{s:permeability}

As stated in the introduction, two variants of the cell swelling model could be considered. Apart from the above, which corresponds to a fixed permeability of the membrane, an alternative model based on aquaporins could be studied. In this section, the model will be adjusted to reflect this, and it will be shown that the gradient flow approach above also works for the resulting problem, with only slightly different results.

\subsection{Adjusting the model}

If it is assumed that the number of aquaporins on the cell membrane is fixed, and the permeance of an aquaporin is also fixed, the total permeance of the membrane is fixed. By radial symmetry, the permeability of the membrane is equal to $\frac{\Perimeter(r(0))}{\Perimeter(r(t))}$, up to a multiplicative constant. This means that the new equation for the normal velocity of the membrane will be
\begin{equation}
	r'(t)=\frac{\Perimeter(r(0))}{\Perimeter(r(t))}\left(-\surfacetension\frac{n-1}{r(t)}+\osmosis u(r(t), t)\right).
\end{equation}
As before, all but one of the parameters can be made equal to $1$. The one remaining parameter will again be $\diffusion$. In this case, however, the scaling factor for $t$ also depends on $r(0)$.

From a modeling point of view, there is only one reasonable way to adapt the gradient flow approach to this new situation: since permeability basically determines how hard it is to move the membrane, the metric for the $r$-coordinate should be modified. Remembering that $\setdist$ arose from minimizing 
\[
	\int_0^1\|v_t\|_{L^2(\partial B_{r(t)})}\ud{t},
\]
it seems reasonable to integrate some other norm of $v_t$. It turns out that choosing the reciprocal of the permeability, that is, $\Perimeter$ as a weight in the $L^2$-norm of $v_t$ is the right choice. Therefore, define
\begin{equation}
	\hat\setdist(r_0, r_1)
		:=\int_{(r_0, r_1)}\Perimeter(\rho)\ud{\rho}
		=\Lebesgue^n(B_{r_1}\triangle B_{r_0})
		=\omega_n\left|r_1^n-r_0^n\right|.
\end{equation}
Obviously, the metric for $\Cells\times\Profiles$ will be the same combination of $\hat\setdist$ and $W_2$ as before:
\begin{equation}
	\hat\varrho((r, u), (s, w)):=\left(\hat\setdist^2(r, s)+\frac{W_2^2(u, w)}{\diffusion}\right)^\frac{1}{2}.
\end{equation}

\subsection{Implications of changing the metric}

Note that the new definition does not change the topology on $\Cells$, which means that most of the results still hold. Some of the finer properties, however, have changed. In this section, the results from the previous sections will be adapted to reflect the new model.

First of all, since the metric has changed, the isometry from $\Cells$ to $[0, +\infty)$ has changed as well. The isometry now reads
\begin{equation}
	\hat\iota_n:\Cells\to[0, +\infty):r\mapsto\omega_n r^n.
\end{equation}
All other results from Section \ref{s:GMM} are the same.

The metric derivative $|.'|_{\hat\setdist}$ is of course also different. Absolute continuity in $[0, +\infty)$ and $\Cells$ are of course still equivalent, but
\begin{equation}
	|r'|_{\hat\setdist}(\tau)=\Perimeter(r(\tau))\left|\diff[r(\tau)]{\tau}\right|,
\end{equation}
for an absolutely continuous curve $\tau\mapsto r(\tau)$. The obvious adaptation of Lemma \ref{th:ball_continuity} also holds.

\begin{lem}\label{th:ball_continuity_permeability}
	Let $\tau\mapsto r(\tau)$ be a curve in $\Cells$. Then $r(\tau)$ is absolutely continuous if and only if it is continuous and there exists a function $g$ such that $g(\tau)\Perimeter(r(\tau))\in L^1_\loc((0, \infty))$,
	\begin{equation}
		\int_0^\infty\int_0^{r(\tau)}\pdiff[\psi(\rho, \tau)]{\tau}\Perimeter(\rho)\ud{\rho}\ud{t}
			=-\int_0^\infty g(\tau)\psi(r(\tau), \tau)\Perimeter(r(\tau))\ud{\tau}
	\end{equation}
	for all $\psi\in C^\infty_c([0, \infty)\times(0, \infty))$. In this case, $|r'|_{\hat\setdist}(\tau)=|g(\tau)|\Perimeter(r(\tau))$ for almost every $\tau>0$.
\end{lem}

As before, the isometry $\hat\iota_n$ can be used to characterize the geodesics of $\hat\setdist$:
\begin{equation}
	r(\tau)=\left((1-\tau)r(0)^n+\tau r(1)^n\right)^\frac{1}{n}\ge(1-\tau)r(0)+\tau r(1),
\end{equation}
which means that combining geodesics for $r$ and $u$ will give a geodesic for $\hat\varrho$.

A more substantial difference is the change in $\lambda$-convexity of $\Perimeter$. By the same argument as before, $\lambda$-convexity of $\Perimeter$ is now equivalent to $\lambda$-convexity of the map
\begin{equation}
	s\mapsto\hat\theta_n r^{n-1}{n},
\end{equation}
where $\theta_n>0$ is a constant. Before, $\Perimeter$ was $0$-convex for $n\ge 3$, but this is not the case anymore: for every $n$, the situation is similar to the situation $n=2$ before: $\Perimeter$ is \emph{not} $\lambda$-convex globally, but only $\lambda$-convex with $\lambda<0$ on sublevels of $\Energy$.

Together with the observation that $s\mapsto\frac{1}{2}\hat\setdist^2(r, s)$ is still $1$-convex for all $r$, \cite[Theorems 2.4.15 and 4.0.4]{Ambrosio_Gigli_Savare} still apply. Since $\lambda$ is still negative, no stronger contraction results follow from these theorems.

Of course, the local slope $\partial\Energy$ with respect to $\hat\varrho$ should be different. Since the proof of Theorem \ref{th:local_slope} largely deals with the mass component, it only needs a small modification.

\begin{thm}\label{th:local_slope_permeability}
	Let $(r, u)\in\States$ be given. Then $\Energy$ has finite local slope at $(r, u)$ if and only if $\hat\integrand(u)\in W^{1, 1}(B_r)$ with 
	\begin{equation}
		\frac{\nabla\hat\integrand(u)}{u}\in L^2(\Reals^n, u).
	\end{equation}
	In this case,
	\begin{equation}
		|\partial\Energy|(r, u)
			=\left(\left|\frac{n-1}{r}-\hat\integrand(u(r))\right|^2
				{}+\diffusion\left\|\frac{\nabla\hat\integrand(u)}{u}\right\|^2_{L^2(\Reals^n, u)}\right)^\frac{1}{2}.
	\end{equation}
\end{thm}
\begin{proof}
	The proof is similar to the proof of Theorem \ref{th:local_slope}. The estimate \eqref{eq:EL_inequality} now should be
	\begin{equation}
		\left|\int_{\Reals^n}\hat\integrand(u(x))\Div\boldsymbol{\xi}(x)\ud{x}\right|
			\le|\partial\Energy|(r, u)
					\left(|\psi(r)|^2\Perimeter^2(r)
						{}+\frac{\|\psi\|^2_{L^2((0, \infty), u\Perimeter)}}{\diffusion}\right)^\frac{1}{2}
					{}+|\psi(r)|\Perimeter'(r)
	\end{equation}
	Proceeding as in the proof of Theorem \ref{th:local_slope} now yields
	\begin{multline*}
		\left|\left(\frac{\Perimeter'(r)}{\Perimeter(r)}-\hat\integrand(u(r))\right)\Perimeter(r)\psi(r)
			{}+\int_{B_r}\frac{\nabla\hat\integrand(u(x))}{u(x)}\cdot\boldsymbol{\xi}(x)u(x)\ud{x}\right| \\
		\le|\partial\Energy|(r, u)\left(|\psi(r)|^2\Perimeter^2(r)
			{}+\frac{\|\boldsymbol{\xi}\|^2_{L^2(\Reals^n, u)}}{\diffusion}\right)^\frac{1}{2},
	\end{multline*}
	which, with a duality argument, gives one of the implications.
	
	The converse implication is also shown as in the proof of Theorem \ref{th:local_slope} but with the Cauchy inequality applied differently:
	\begin{equation}\begin{split}
		\Energy(s, w)-\Energy(r, u)
		&	\ge(\Perimeter(s)-\Perimeter(r))
				{}-\int_{\Reals^n}\hat\integrand(u(x))\Div(\boldsymbol{t}_u^w(x)-x)\ud{x} \\
		&	\ge\left(\frac{\Perimeter'(r)}{\Perimeter(r)}-\hat\integrand(u(r))\right)\left(s-r\right)\Perimeter(r)
				{}+\int_{\Reals^n}\nabla\hat\integrand(u(x))\cdot(\boldsymbol{t}_u^w(x)-x)\ud{x} \\
		&	\ge -\left(\left|\frac{n-1}{r}-\hat\integrand(u(r))\right|^2
				{}+\diffusion\left\|\frac{\nabla\hat\integrand(u)}{u}\right\|^2_{L^2(\Reals^n, u)}
				\right)^\frac{1}{2} \\
		&\hspace{1.5cm}		\left(|s-r|^2\Perimeter^2(r)+\frac{W_2^2(u, w)}{\diffusion}\right)^\frac{1}{2},
	\end{split}\end{equation}
	which, after taking limits, again concludes the proof.
\end{proof}

With expressions for the metric derivative and local slope at hand, the proof of Theorem \ref{th:weak_solution} can be repeated. In this new setting it will give
\begin{equation}
	\Perimeter(r(t))r'(t)=-\frac{n-1}{r(t)}+\hat\integrand(u(r(t), t))
\end{equation}
for almost all $t>0$, which means that \eqref{eq:weak_boundary} becomes
\begin{equation}\label{eq:weak_boundary_permeability}
	\int_0^\infty\int_{B_{r(t)}}\pdiff[\psi(x, t)]{t}\ud{x}\ud{t}
		=\int_0^\infty\left(\frac{n-1}{r(t)}-\hat\integrand(u(r(t), t))\right)\psi(r(t), t)\ud{t}.
\end{equation}
The equation for $u$ is, as expected, not changed.

\section*{Acknowledgements}
The author gratefully acknowledges the advice, insight and criticism of Prof. dr. Mark Peletier. Without his support, this paper would not have been written. The author would also like to thank Joost Hulshof, Georg Prokert and Matthias R\"oger for their criticism and interest during the writing of this paper. Finally, the author would like to thank the anonymous referee for pointing out some implications of \cite[Theorems 2.4.15 and 4.0.4]{Ambrosio_Gigli_Savare} and offering many useful suggestions.

\end{document}